\documentclass{article}
\usepackage[a4paper, total={6in, 8in}]{geometry}
\usepackage{authblk}
\usepackage{amsfonts}
\usepackage{amsmath}
\usepackage{amssymb}
\usepackage{bm}
\usepackage{cancel}
\usepackage{braket}
\usepackage{ifthen}
\usepackage[most]{tcolorbox}
\usepackage{amsthm}
\newtheorem{theorem}{Theorem}
\newtheorem{definition}[theorem]{Definition}
\newtheorem{lemma}[theorem]{Lemma}
\newtheorem{proposition}[theorem]{Proposition}

\usepackage{tikz}
\usetikzlibrary{shapes}
\usetikzlibrary{perspective}
\usetikzlibrary{positioning}
\usepackage{hyperref}

\newcommand*{\ol}{\overline}

\newcommand{\nn}{\nonumber}

\numberwithin{theorem}{section}
\numberwithin{equation}{section}

\makeatletter
\DeclareRobustCommand{\rvdots}{%
  \vbox{
    \baselineskip4\p@\lineskiplimit\z@
    \kern-\p@
    \hbox{.}\hbox{.}\hbox{.}
  }}
\makeatother

\title{Multiple commutation relations
of the quantum affine algebra $U_q(\widehat{\mathfrak{gl}}_N)$, nested Bethe vector and the Gelfand-Tsetlin basis}

\author[1]{Allan John Gerrard\footnote{gerrarda@fc.jwu.ac.jp}}
\affil[1]{Department of Mathematics, Physics and Computer Science, Japan Women's University,}

\author[2]{Kohei Motegi\footnote{kmoteg0@kaiyodai.ac.jp}}
\affil[2]{Faculty of Marine Technology, Tokyo University of Marine Science and Technology}

\author[3]{Kazumitsu Sakai\footnote{k.sakai@rs.tus.ac.jp}}
\affil[3]{Department of Physics, Tokyo University of Science}

\begin{document}

\maketitle

\begin{abstract}
We study a certain type of multiple commutation relations of the quantum affine algebra $U_q(\widehat{\mathfrak{gl}}_N)$.
We show that all the coefficients in the multiple commutation relations
between the $L$-operator elements
are given in terms of the trigonometric weight functions for the vector representation,
independent of the representation of the $L$-operator.
For rank one case, our proof also gives a conceptual understanding why
the coefficients can also be expressed using the Izergin-Korepin determinants.
As a related result,
by specializing expressions for the universal nested Bethe vector by Pakuliak-Ragoucy-Slavnov,
we also find a construction of the Gelfand-Tsetlin basis for the vector representation
using different $L$-operator elements from the constructions by Nazarov-Tarasov or Molev.
We also present corresponding results for the Yangian $Y_h(\mathfrak{gl}_N)$.
\end{abstract}

\section{Introduction}
Quantum groups \cite{Drinfeld,Jimbo,RTF} are Hopf algebra deformations of universal enveloping algebras of Lie algebras,
introduced in connection with integrable systems and the Yang–Baxter equation
\cite{Bethe,PS,Baxterbook,KBI,Slavnovbook}
in the 1980s.
Since then, numerous works have been devoted to their structural theory and applications across representation theory, low-dimensional topology, and mathematical physics.
Nevertheless, even certain naive questions on quantum groups still remain to 
be difficult or need conceptual understandings.
In this paper, we deal with one of such,
the multiple commutation relations.
The defining relations of the quantum groups are given as
commutation relations between two $L$-operator elements.
A naive question is then what are the explicit commutation relations
if there are more types and multiple $L$-operator elements.
Investigating multiple commutation relations and multiple actions on vectors
is a fundamental problem and are important for the study
of correlation functions of quantum integrable models \cite{KBI,Slavnovbook,KMST}.
It has also been important for the study of stochastic integrable models \cite{ABPW} and
the three-dimensional partition functions \cite{IMO}. As for higher rank quantum groups and Yangians,
significant progress began to emerge recently. See \cite{PRS,HLPRS,LP} for example.
This is also aligned with the progress of algebraic understanding of the nested Bethe vectors
and higher rank weight functions
\cite{Schultz,Reshetikhin,TVleningrad,Mimachi,TVasterisque,FKPR,BPR,TVsigma,FM,GR,Slavnov,BW,GMS,KT}.
The higher rank trigonometric weight functions are extended to the elliptic case in recent years
\cite{KonnoJintone,KonnoJinttwo,RTV,KM}.

In this paper, we investigate  a certain type of multiple commutation relations
between $L$-operator elements.
We determine the coefficients of the summands 
using specializations of higher rank trigonometric weight functions.
The type of proof as well as the
multiple commutation relations we present in this paper
is different from the ones
studied in previous researches,
which is typically proved by induction on the number of $L$-operator elements.
The strategy of our proof is as follows.
From the basic commutation relations,
we note the coefficients
of the summands appearing in the multiple commutation relations
are independent of the representation of the $L$-operator.
We show that we can extract each coefficient 
by taking an appropriate tensor product of vector representation
for the $L$-operator, and identify as a certain type
of partition functions constructed from the standard $U_q(\widehat{\mathfrak{gl}}_N)$ $R$-matrix on a rectangular grid.
This type of partition functions can be related with
the type of partition functions which represent the off-shell nested Bethe wavefunctions
whose explicit forms are given by the trigonometric weight functions,
and finally all the coefficients are determined as specializations of the trigonometric weight functions
with certain overall factors multiplied.
The argument we present is conceptual,
and also gives an understanding why the coefficients can be expressed using the
Izergin-Korepin determinants for rank one case.

From this study on multiple commutation relations, we note there is a construction
of the Gelfand-Tsetlin basis for the quantum group $U_q(\widehat{\mathfrak{gl}}_N)$
for the case of the vector representation,
which is different from the constructions by Nazarov-Tarasov \cite{NT}
using quantum minors or Molev \cite{Molev} using different $L$-operator elements,
which may not have been written down previously.
We take an indirect approach, and by specializing two expressions for the universal Bethe vectors
by Pakuliak-Ragoucy-Slavnov \cite{PRS},
we give the precise relation, i.e. determine the proportional constants
between
the Gelfand-Tsetlin basis corresponding to the
construction by Molev applied to the vector representation
and another construction which uses different $L$-operator elements.

This paper is organized as follows.
In the next section, we review the quantum affine algebra $U_q(\widehat{\mathfrak{gl}}_N)$,
two basic partition functions associated with the algebra, and their relation with special functions.
In Section 3, we present and give a proof of the multiple commutation relations.
In Section 4, we present a construction of the Gelfand-Tsetlin basis for the tensor product
of vector representation.
In Section 5, we present corresponding results for the Yangian $Y_h(\mathfrak{gl}_N)$. \\

In this paper,
sets of variables and products of functions frequently appear.
We adopt the conventions which are basically the same as the ones used in, for example, \cite{PRS}.
Sets of parameters will be denoted using an overline $\ol{w} = \{w_i | i \in I\}$, where $I$ is some index set. 
We denote the cardinality of the set $\ol{w}$ by $|\ol{w}|$.
For any function, say $f$,
\begin{equation} \label{product-notation}
	f(\ol{w}) = \prod_{w_i\in \ol{w}} f(w_i).
\end{equation}
Note that this will also be used for multivariate functions and binary operations, e.g., $\ol{u}-\ol{v} = \prod_{u_i \in \ol{u}} \prod_{v_j \in \ol{v}} (u_i-v_j)$.
As for operators, we make use of this notation for commuting operators only. 
We will use the notation $\{\ol{u}, \ol{v}\}$ for the union of sets $\ol{u} \cup \ol{v}$. 

We write
\[
\{ \overline{u}^1,\dots,\overline{u}^N \}\;\mapsto\;\{ \overline{v}^1,\dots,\overline{v}^N \}
\]
for a partition of the set $\{\overline{u}^1,\dots,\overline{u}^N\}$ into subsets
$\{ \overline{v}^1,\dots,\overline{v}^N \}$ with $|\overline{v}^j|=|\overline{u}^j|$ for all $j$.
The notation
\[
\sum_{\{\overline{u}^1,\dots,\overline{u}^N\}\mapsto\{\overline{v}^1,\dots,\overline{v}^N\}}
\]
denotes the sum over all such partitions.
For example, 
when $\overline{u}^1=\{ a \}$, $\overline{u}^2=\{ b, c \}$,

\begin{align}
\sum_{\{ \overline{u}^1, \overline{u}^2 \} \mapsto \{ \overline{v}^1,
	\overline{v}^2 \} }
	f(\overline{u}^1, \overline{v}^1)g(\overline{u}^1, \overline{v}^2)
=f(a,a)g(a,b)g(a,c)+f(a,b)g(a,a)g(a,c)+f(a,c)g(a,a)g(a,b),
\end{align}
since there are three cases: (i)
$\overline{v}^1=\{ a \}$, $\overline{v}^2=\{ b, c \}$ (ii)
$\overline{v}^1=\{ b \}$, $\overline{v}^2=\{ a, c \}$ (iii)
$\overline{v}^1=\{ c \}$, $\overline{v}^2=\{ a, b \}$.

We denote the symmetric group, the group of all permutations of a set with \( n \) elements, by \( S_n \).
We also denote $m$ consecutive occurrences of the number $i$ by $i^m$.
This type of notation appears in the description of the (dual) orthonormal basis, for example $e_{i^m} = \underbrace{e_{i} \otimes \cdots \otimes e_{i}}_{m \text{ times}}$.

As for the Gelfand-Tsetlin basis, rather than using the Gelfand-Tsetlin patterns found in \cite{NT,Molev},
we use the convention of \cite{Mimachi,TVasterisque,KonnoJintone,KonnoJinttwo,RTV} that is suited for the description of the basis corresponding to the tensor product of vector representations.
Typically, what we use is a partition $J=(J_1,\dots,J_N)$ such that
\begin{align} \label{partitionJ}
J_1 \cup \cdots \cup J_N=\{1,\dots,n \}, \ \ \ J_j \cap J_k =\varnothing \ (j \neq k).
\end{align}
The $J_j$ are permitted to be empty.
Typical (sub)sets that appear are
$\overline{w}=\{ w_1,\dots,w_n  \}$
and $\overline{w}_{J_j}=\{ w_k \ | \ k \in J_j  \}$ ($j=1,\dots,N$).

Finally, we remark that the parameter $q$ for the quantum affine algebra
$U_q(\widehat{\mathfrak{gl}}_N)$ and $h$ for the Yangian
$Y_h(\mathfrak{gl}_N)$ are both assumed to be generic non-zero complex numbers.

\section{Quantum affine algebra \texorpdfstring{$U_q(\widehat{\mathfrak{gl}}_N)$}{UqGLN} and basic partition functions}
Let $V$ be a complex $N$-dimensional vector space and denote its standard orthonormal basis
by $e_i$, $i=1,\dots,N$. We denote the dual of $V$ by $V^*$ and
the dual basis by $e_i^*$, $i=1,\dots,N$, which satisfy $e_i^* e_j=\delta_{ij}$.
Here, $\delta_{ij}$ is the Kronecker delta: $\delta_{ij}=1$ if $i=j$ and $\delta_{ij}=0$ otherwise.
We introduce the standard matrix units $E_{ij}$, $i,j=1,\dots,N$ as $E_{ij}e_k=\delta_{jk}e_i$. 

We introduce the trigonometric $R$-matrix $R(u,v)$ $\in \mathrm{End}(V \otimes V)$
\begin{align}
	R(u,v)=&(qu-q^{-1}v) \sum_{1 \leq i \leq N} E_{ii}\otimes E_{ii}
	+ \sum_{1 \leq i < j \leq N} (u-v)(E_{ii}\otimes E_{jj} + E_{jj} \otimes E_{ii}) 
\nonumber	\\
&+ \sum_{1 \leq i < j \leq N} ((q-q^{-1})v E_{ij} \otimes E_{ji} + (q-q^{-1})u E_{ji} \otimes E_{ij}).
\label{trigRverone}
\end{align}

Denoting $\displaystyle R(u,v)e_i \otimes e_j=\sum_{k,\ell=1}^N [R(u,v)]_{ij}^{k \ell}e_k \otimes e_\ell$,
the non-zero $R$-matrix elements are
\begin{align}
[R(u,v)]_{ii}^{ii}&=qu-q^{-1}v, \ \ \ i=1,\dots,N, \\
[R(u,v)]_{ij}^{ij}&=u-v, \ \ \ i, j=1,\dots,N, \ \ \ i \neq j, \\
[R(u,v)]_{ji}^{ij}&=(q-q^{-1})v, \ \ \ 1 \le i < j \le N, \\
[R(u,v)]_{ij}^{ji}&=(q-q^{-1})u, \ \ \ 1 \le i < j \le N.
\end{align}

\begin{figure}
	\centering
	\includegraphics{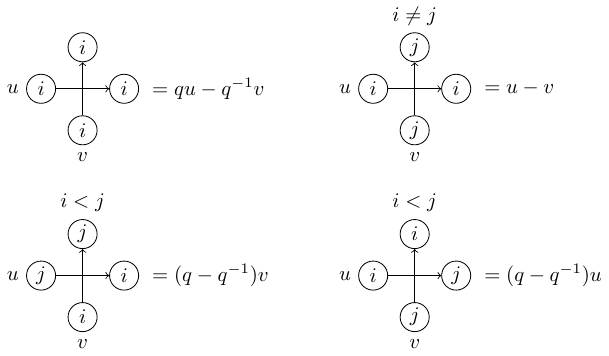}
	\caption{Non-zero matrix elements of the trigonometric $R$-matrix.} \label{fig:trigRmatrix}
\end{figure}

The $R$-matrix elements $[R(u,v)]_{ij}^{k \ell}$ satisfy
\begin{equation} \label{icerule}
	[R(u,v)]_{ij}^{k \ell}=0 \qquad \text{unless } (i=k, \, j=\ell) \text{ or } (i=\ell, \, j=k).
\end{equation}
This is more restrictive than the ice rule, where
$[R(u,v)]_{ij}^{k \ell}$=0 unless $i+j=k+\ell$, which is obeyed by the $R$-matrices of higher spin representations of $U_q(\widehat{\mathfrak{sl}}_2)$.
We will refer to the property \eqref{icerule} as \emph{color conservation}.


We use graphical descriptions in this paper.
See 
Figure~\ref{fig:trigRmatrix}
for the trigonometric $R$-matrix.

\begin{figure}
	\centering
	\includegraphics{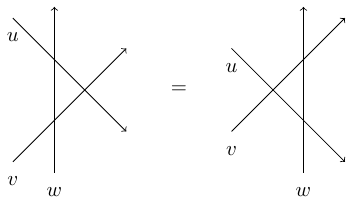}
	\caption{Graphical description of the Yang-Baxter equation.} \label{fig:ybe}
\end{figure}

\begin{figure}
	\centering
	\includegraphics{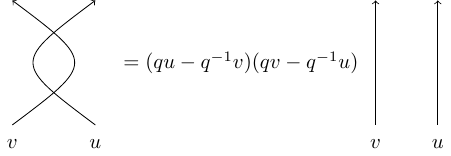}
	\caption{Graphical description of the unitarity relation.} \label{fig:unitarity}
\end{figure}

The trigonometric $R$-matrix satisfies the Yang-Baxter relation 
(Figure~\ref{fig:ybe})
\begin{equation}
\label{eq:YBE-spectral}
 R_{12}(u,v)R_{13}(u,w)R_{23}(v,w)
 = R_{23}(v,w)R_{13}(u,w)R_{12}(u,v) \in \mathrm{End}(V \otimes V \otimes V),
\end{equation}
and the unitarity relation (Figure~\ref{fig:unitarity}),
\[
R_{12}(u,v) \, R_{21}(v,u) =(qu-q^{-1}v)(qv-q^{-1}u) \mathrm{I} \otimes \mathrm{I}  \in \mathrm{End}(V \otimes V),
\]
where $\mathrm{I}$ is the identity matrix acting on $V$.

The quantum affine algebra $U_q(\widehat{\mathfrak{gl}}_N)$ is defined as the unital associative algebra generated by the formal series
\[
T(u) = \sum_{i,j=1}^N E_{ij} \otimes T_{ij}(u), 
\qquad 
\bar T(u) = \sum_{i,j=1}^N E_{ij} \otimes \bar T_{ij}(u),
\]
where
\[
T_{ij}(u) = \sum_{r \ge 0} T_{ij}[r] \, u^{-r}, 
\qquad 
\bar T_{ij}(u) = \sum_{r \ge 0} \bar T_{ij}[r] \, u^{r},
\]
subject to the following relations:
\begin{align}
R(u,v) \, T_1(u) T_2(v) &= T_2(v) T_1(u) \, R(u,v), \label{eq:RTT++} \\[6pt]
R(u,v) \, \bar T_1(u) \bar T_2(v) &= \bar T_2(v) \bar T_1(u) \, R(u,v), \label{eq:RTT--} \\[6pt]
R(u,v) \, T_1(u) \bar T_2(v) &= \bar T_2(v) T_1(u) \, R(u,v). \label{eq:RTT+-}
\end{align}
Here, $T_1(u) = T(u) \otimes 1$, $T_2(v) = 1 \otimes T(v)$ (and similarly for $\bar T$).
The above definition is for the case where the central charge is equal to zero; for the generic central charge case see \cite{DF,KK} for example.
In this paper, we refer to
$T(u)$ as the {\it L-operator}
and
$T_{ij}(u)$, $i,j=1,\dots,N$ as the {\it L-operator elements}
(also called the {\it monodromy matrix} and the {\it monodromy matrix elements} in the literature).

In this paper we consider the relation \eqref{eq:RTT++},
which is a collection of commutation relations between $T_{ij}(u)$, $i,j=1,\dots, N$.
Here we write down the relations which we consider in more depth
\begin{align}
T_{ik}(u)T_{ij}(v)&=\frac{qv-q^{-1}u}{v-u}T_{ij}(v)T_{ik}(u)-\frac{u(q-q^{-1})}{v-u}T_{ij}(u)T_{ik}(v), \ \ \  k<j, \label{fundcommone} \\
[T_{ij}(u),T_{ij}(v)]&=0, \label{fundcommtwo}
\end{align}
for $i,j,k=1,\dots,N$.

\begin{figure}
	\centering
	\includegraphics{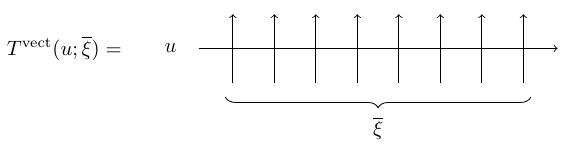}
	\includegraphics{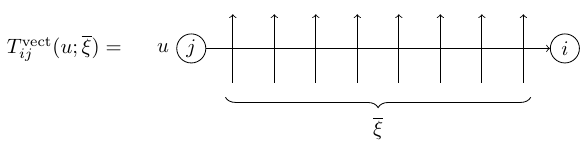}
	\caption{Graphical description of $T^{\mathrm{vect}}(u;\overline{\xi})$ and $T_{ij}^{\mathrm{vect}}(u;\overline{\xi})$.} \label{fig:tvect}
\end{figure}

The most natural representation of $T(u)$ which we denote by
$T^{\mathrm{vect}}(u;\overline{\xi})$ is defined as
\begin{align}
T^{\mathrm{vect}}(u;\overline{\xi})=R_{0n}(u,\xi_n) \cdots R_{02}(u,\xi_2) R_{01}(u,\xi_1),
\end{align}
where $n=|\overline{\xi}|$.
Note that here we do not use the notation defined in \eqref{product-notation}, as the $R$-matrices do not mutually commute.
$T^{\mathrm{vect}}(u;\overline{\xi})$ acts on $V_0 \otimes V_1 \otimes V_2 \otimes \cdots \otimes V_n$,
where each $V_i$ is a copy of $V$.
$V_0$ is called the auxiliary space and $V_i$ ($i=1,\dots,n$) are called quantum spaces.
Applying the Yang-Baxter relation \eqref{eq:YBE-spectral} repeatedly, we note that
$T^{\mathrm{vect}}(u;\overline{\xi})$ satisfies \eqref{eq:RTT++}.
We denote the $T_{ij}(u)$ corresponding to this representation by $T_{ij}^{\mathrm{vect}}(u;\overline{\xi})$
or $T_{ij}(u;\overline{\xi})$ for short. $T_{ij}(u;\overline{\xi})$ acts on $V_1 \otimes V_2 \otimes \cdots \otimes V_n$.
We denote
the standard orthonormal (dual) basis of $V_1 \otimes V_2 \otimes \cdots \otimes V_n$ by
 $e_{i_1 i_2 \dots i_n}=e_{i_1} \otimes e_{i_2} \otimes \cdots \otimes e_{i_n}$ 
($e^*_{i_1 i_2 \dots i_n}=e^*_{i_1} \otimes e^*_{i_2} \otimes \cdots \otimes e^*_{i_n}$), $i_1,i_2,\dots,i_n=1,\dots,N$.
See Figure~\ref{fig:tvect} for a
graphical description of $T^{\mathrm{vect}}(u;\overline{\xi})$
and $T_{ij}^{\mathrm{vect}}(u;\overline{\xi})$.
To each copy of the vector space $V$, a spectral variable is assigned.
For $T^{\mathrm{vect}}(u;\overline{\xi})$,
we assign $u$ to $V_0$, and
$\xi_i$ is assigned to each of the quantum spaces $V_i$ ($i=1,2,\dots,n$).

\color{red}













	

\color{black}

\begin{figure}
	\centering
	\includegraphics{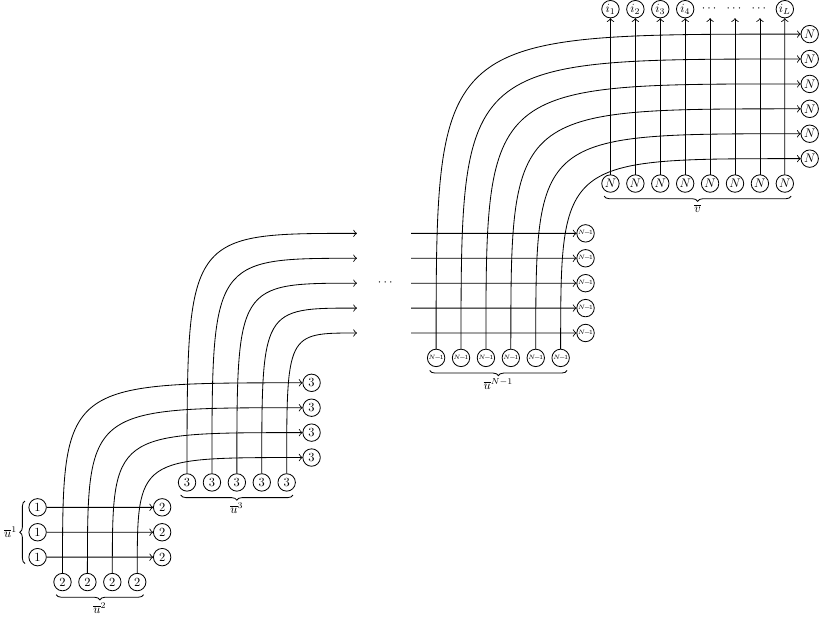}
	\caption{The partition function. } \label{fig:partitionfn}
\end{figure}

We introduce a class of partition functions graphically represented by
Figure~\ref{fig:partitionfn},
which consist of
$(N-1)$ layers. We call the layer in the bottom-left the first layer, the layer northeast of it the second layer, and so
on. The $j^\text{th}$ layer consists of horizontal lines and vertical lines which represent vector spaces; we call them the auxiliary spaces and quantum spaces in the $j^\text{th}$ layer respectively.

For $j=1,\dots,N-1$,
we denote the set of spectral variables in the auxiliary spaces in the $j^\text{th}$ layer
by $\overline{u}^j
=\{u_1^{j},\dots,u_{k_j}^{j}  \}$ ($k_j=|\overline{u}^j|$). The ordering of the variables for each set
can be arbitrary
since the type of partition functions under consideration is symmetric
with respect to the permutation of the variables due to the Yang-Baxter relation \eqref{eq:YBE-spectral}.

The set of spectral variables of the quantum spaces in the $(N-1)^\text{th}$ layer is denoted
by $\overline{v}
=\{v_1,\dots,v_{L}  \}$ ($L=|\overline{v}|$).
As for this set of variables,
the type of partition functions  is not symmetric in general.
To each quantum space with spectral variable in $\overline{v}$, we assign a color, 
which represents contraction by the corresponding basis vector. In Figure~\ref{fig:partitionfn}, this is denoted by a circle on top of the vertical line (quantum space).
We call the quantum space to which the spectral variable $v_j$ is assigned the {\it $j^\text{th}$ quantum space}.
We call the place to which a color is assigned
in the $j^\text{th}$ quantum space the {\it $j^\text{th}$ coordinate}, running from $1,\dots,L$.
We denote the color assigned to the $j^\text{th}$ coordinate by $i_j$.
We denote the set of colors $(i_1,\dots,i_L)$ by
$\boldsymbol{I}$.
For $\boldsymbol{I}$,
let $\boldsymbol{I}^j \subset \{1,\dots,L \}$ denote the coordinates colored by $1,\dots,j$.
Note that $|\boldsymbol{I}^j|=k_j$.
For $\boldsymbol{I}^j \subset \boldsymbol{I}^{j+1}$,
we introduce $\widetilde{\boldsymbol{I}}^j$ as follows.
We relabel the coordinates of $\boldsymbol{I}^{j+1}$ to $\{ 1,\dots,k_{j+1} \}$, preserving order.
Accordingly, as a subset, $\boldsymbol{I}^j $ is mapped to a subset of
$\{ 1,\dots,k_{j+1} \}$ which we define as $\widetilde{\boldsymbol{I}}^j=\{ \widetilde{I}_1^j
< \widetilde{I}_2^j< \cdots < \widetilde{I}_{k_j}^j
 \}$; see Figure~\ref{fig:label-summary}.
 
\begin{figure}
	\centering
	\includegraphics{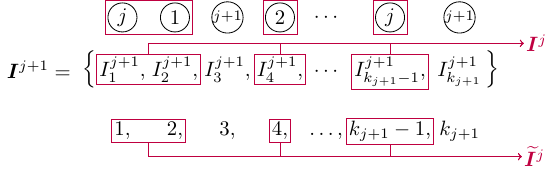}
	\caption{Graphical explanation of the construction of $\widetilde{\boldsymbol{I}}^j$.} \label{fig:label-summary}
\end{figure}

We denote the partition function by $\psi(\overline{u}^1,\dots,\overline{u}^{N-1} | \overline{v} | \bm I )$.
This corresponds to a graphical description of a version
of the off-shell nested Bethe wavefunctions.

We introduce the trigonometric weight functions.

\begin{definition}
The trigonometric weight functions are defined as
\begin{align} \label{trigonometric-wavefunction}
        &W (\overline{u}^1,\dots,\overline{u}^{N-1} | \overline{v} | \bm I )        
       := \sum_{\sigma_1 \in S_{k_1}} \cdots \sum_{\sigma_{N-1} \in S_{k_{N-1}}} \nonumber \\
        &\prod_{p=1}^{N-2} 
\Bigg\{
            \prod_{a=1}^{k_p} \Bigg(
\prod_{i=1}^{\widetilde{I}_{a}^{p}-1} 
 \left(u^{p}_{\sigma_p(a)}- u^{p+1}_{\sigma_{p+1}(i)}   \right) \times (q-q^{-1})u^{p}_{\sigma_p(a)} \times
                \prod_{i=\widetilde{I}_{a}^{p}+1}^{ k_{p+1}}
                    \left(q u^{p}_{\sigma_p(a)}-q^{-1} u^{p+1}_{\sigma_{p+1}(i)} \right)
            \Bigg) \nonumber \\
&\times
            \prod_{a<b}^{k_p} 
                \frac{q^{-1} u^{p}_{\sigma_p(a)}-q u^{p}_{\sigma_p(b)}}{u^{p}_{\sigma_p(a)}-u^{p}_{\sigma_p(b)}}
        \Bigg\}
\nonumber        \\ 
&\times 
        \prod_{a=1}^{k_{N-1}} \left(
            \prod_{i=1}^{I_a^{N-1}-1} 
                \left(u_{\sigma_{N-1}(a)}^{N-1}-v_i^{N-1}\right) \times (q-q^{-1})u_{\sigma_{N-1}(a)}^{N-1} \times
            \prod_{i=I_a^{N-1}+1}^{L} 
                \left(qu_{\sigma_{N-1}(a)}^{N-1}-q^{-1}v_i^{N-1} \right)
        \right)
     \nonumber   \\ 
&\times 
        \prod_{a<b}^{k_{N-1}} 
            \frac{q^{-1} u^{N-1}_{\sigma_{N-1}(a)}- q u^{N-1}_{\sigma_{N-1}(b)}}{u^{N-1}_{\sigma_{N-1}(a)}-u^{N-1}_{\sigma_{N-1}(b)}}.
    \end{align}   

\end{definition}

There is
the following correspondence between the partition functions and the trigonometric weight functions.
\begin{theorem} \label{offshellbethetrigonometricweight}
The following holds:
\begin{align}
\psi(\overline{u}^1,\dots,\overline{u}^{N-1} | \overline{v} | \bm I ) =
W (\overline{u}^1,\dots,\overline{u}^{N-1} | \overline{v} | \bm I ) .
\end{align}
\end{theorem}

See
 \cite{PRS,FKPR,TVsigma,FM,BW,GMS,KT},
for example, for various types of proofs, extensions, other forms
and related topics.

\begin{figure}
	\centering
	\includegraphics{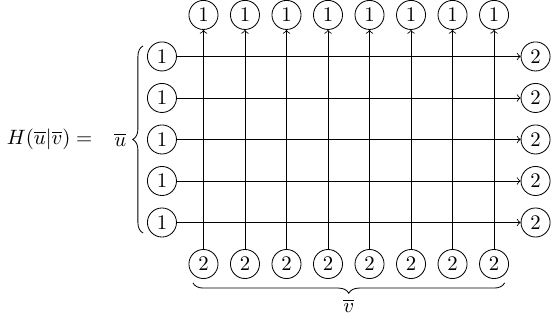}
	\caption{The domain wall boundary partition function.} \label{fig:dwpf}
\end{figure}

We also introduce the domain wall boundary partition functions (Figure~\ref{fig:dwpf}).
\begin{align}
&H(\overline{u}|\overline{v}) 
:=
e_{1^{|\overline{u}|}}^*
T_{21}(\overline{u};\overline{v}) 
e_{2^{|\overline{u}|}},
\end{align}
where $|\overline{u}|=|\overline{v}|$.
Here we recall $e_{2^{|\overline{u}|}}$ refers to the vector $\underbrace{e_{2} \otimes \dots \otimes e_{2}}_{|\overline{u}| \text{ times}}$. 
The domain wall boundary partition functions are
symmetric with respect to the set of variables $\overline{u}$ as well as
$\overline{v}$, due to the Yang-Baxter relation.

The 
Izergin-Korepin determinant is defined as
\begin{align}
K(\overline{u}|\overline{v})
&:= 
		\frac{\prod_{1\leq i,j \leq n}(qu_i-q^{-1}v_j)(u_i-v_j)}{\prod_{1\leq i <j \leq n}(u_i-u_j)(v_j-v_i)}
		\det_{1 \le i,j \le n} \left[\frac{(q-q^{-1})u_i}{(qu_i-q^{-1}v_j)(u_i-v_j)}\right] \nonumber \\
		&=
		\frac{1}{\prod_{1\leq i <j \leq n}(u_i-u_j)(v_j-v_i)}
		\det_{1 \le i,j \le n} \left[
		(q-q^{-1})u_i
		\prod_{\substack{k=1 \\ k \neq j}}^n
		(qu_i-q^{-1}v_k)(u_i-v_k)
		\right],
		\label{introductionDW}
\end{align}
for $\overline{u}=\{u_1,u_2,\dots,u_n  \}$, $\overline{v}=\{v_1,v_2,\dots,v_n \}$.

The domain wall boundary partition functions
can be expressed as the Izergin-Korepin determinant, which is a classical fact \cite{Korepin,Izergin}.
\begin{theorem}
The following holds:
\begin{align}
H(\overline{u}|\overline{v}) = K(\overline{u}|\overline{v}).
\label{domainwallpartitionIKdet}
\end{align}
\end{theorem}

\section{Multiple commutation relations}

We derive the following multiple commutation relations 
for the quantum affine algebra $U_q(\widehat{\mathfrak{gl}}_N)$.
\begin{theorem} \label{thmmultiplecommutationrelations}
The following commutation relation holds:
\begin{align}
&T_{N1}(\overline{u}^1) T_{N2}(\overline{u}^2) \cdots T_{NN}(\overline{u}^N)
=\sum_{\{ \overline{u}^1, \overline{u}^2,\dots,\overline{u}^N \} \mapsto \{ \overline{v}^1,
\overline{v}^2,\dots,\overline{v}^N \} } 
\frac{1}{\prod_{1 \le j < k \le N} (\overline{v}^k-\overline{v}^{j})(q \overline{v}^j-q^{-1} \overline{v}^{k-1})}
\nonumber \\
\times&
\frac{\prod_{j=1}^{N-1} (q \overline{u}^N-q^{-1} \overline{v}^j)}{
\prod_{j=1}^{N-1} (q \overline{u}^j-q^{-1} \overline{v}^N)
\prod_{\ell=1}^{N-2}
\prod_{j=1}^\ell \prod_{k=1}^{\ell+1} (q \overline{u}^j-q^{-1} \overline{u}^k)}
\nonumber \\
\times&W(
\overline{u}^1,\overline{u}^1 \cup \overline{u}^2,\dots,\overline{u}^1 \cup \cdots \cup \overline{u}^{N-1}
|\overline{v}^1,\overline{v}^2,\dots,\overline{v}^{N-1},\overline{v}^N|1^{|\overline{u}^1|},2^{|\overline{u}^2|},\dots,(N-1)^{|\overline{u}^{N-1}|},N^{|\overline{u}^{N}|} )
\nonumber \\
\times&
T_{NN}(\overline{v}^N) \cdots T_{N2}(\overline{v}^2)  T_{N1}(\overline{v}^1), 
\end{align}
where
$W(
\overline{u}^1,\overline{u}^1 \cup \overline{u}^2,\dots,\overline{u}^1 \cup \cdots \cup \overline{u}^{N-1}
|\overline{v}^1,\overline{v}^2,\dots,\overline{v}^{N-1},\overline{v}^N|1^{|\overline{u}^1|},2^{|\overline{u}^2|},\dots,(N-1)^{|\overline{u}^{N-1}|},N^{|\overline{u}^{N}|} )$ are specializations of the trigonometric weight functions
\eqref{trigonometric-wavefunction}.
\end{theorem}

Recall that
\[
\sum_{\{\overline{u}^1,\dots,\overline{u}^N\}\mapsto\{\overline{v}^1,\dots,\overline{v}^N\}}
\]
denotes the sum over all partitions 
$\{\overline{u}^1,\dots,\overline{u}^N\}\mapsto\{\overline{v}^1,\dots,\overline{v}^N\}$
satisfying
$|\overline{v}^j|=|\overline{u}^j|$, $j=1,\dots,N$.

The proof of Theorem \ref{thmmultiplecommutationrelations}
is lengthy and follows from combining
Proposition \ref{keypropositionforproof}
and Proposition \ref{keyrelationforproof}.

Proving 
Proposition \ref{keypropositionforproof}
is the first step,
which essentially identifies every coefficient of the summands
as a certain partition function on a rectangular grid.

\begin{proposition} \label{keypropositionforproof}
The following commutation relation holds:
\begin{align}
&T_{N1}(\overline{u}^1) T_{N2}(\overline{u}^2) \cdots T_{NN}(\overline{u}^N)
=\sum_{\{ \overline{u}^1, \overline{u}^2,\dots,\overline{u}^N \} \mapsto \{ \overline{v}^1,
\overline{v}^2,\dots,\overline{v}^N \} } 
\frac{1}{\prod_{1 \le j < k \le N} (\overline{v}^k-\overline{v}^{j})(q \overline{v}^j-q^{-1} \overline{v}^{k-1})}
\nonumber \\
\times&
\prod_{j=1}^{N-1} (q \overline{u}^N-q^{-1} \overline{v}^j)
H(\overline{u}^1, \overline{u}^2,\dots,\overline{u}^{N-1}|\overline{v}^1,\overline{v}^2,\dots,\overline{v}^{N-1})
T_{NN}(\overline{v}^N) \cdots T_{N2}(\overline{v}^2)  T_{N1}(\overline{v}^1), 
\end{align}
where
\begin{align}
&H(\overline{u}^1, \overline{u}^2,\dots,\overline{u}^{N-1}|\overline{v}^1,\overline{v}^2,\dots,\overline{v}^{N-1}) \nonumber \\
=&
e_{1^{|\overline{u}^1|},2^{|\overline{u}^2|},\dots,(N-1)^{|\overline{u}^{N-1}|}}^*
T_{N1}(\overline{u}^1;\overline{v}^1,\overline{v}^2,\dots,\overline{v}^{N-1}) 
T_{N2}(\overline{u}^2;\overline{v}^1,\overline{v}^2,\dots,\overline{v}^{N-1}) 
\nonumber \\
&\times \cdots \times
T_{NN-1}(\overline{u}^{N-1};\overline{v}^1,\overline{v}^2,\dots,\overline{v}^{N-1}) 
e_{N^{|\overline{u}^1|+|\overline{u}^2|+\cdots+|\overline{u}^{N-1}|}}.
\label{commcoeffratpart}
\end{align}
Here,
$T_{jk}(u;\overline{v}^1,\overline{v}^2,\dots,\overline{v}^{N-1})$ denotes
\begin{align}
T_{jk}(u;\overline{v}^1,\overline{v}^2,\dots,\overline{v}^{N-1})
=T_{jk}^{\mathrm{vect}}(u;\overline{\xi}),
\end{align}
where
$\{ \xi_1,\dots,\xi_{|\overline{v}^1|} \}=\overline{v}^1, \{ \xi_{|\overline{v}^1|+1},\dots,\xi_{
|\overline{v}^1|+|\overline{v}^{2}|} \}=\overline{v}^{2},
\dots, \{ \xi_{|\overline{v}^1|+\cdots+|\overline{v}^{N-2}|+1},\dots, \xi_{|\overline{v}^1|+\cdots+|\overline{v}^{N-1}|}
\}=\overline{v}^{N-1}$.

\end{proposition}

\begin{proof}
We can see that in principle,
we can commute $L$-operator elements into the following form
\begin{align}
&T_{N1}(\overline{u}^1) T_{N2}(\overline{u}^2) \cdots  T_{NN}(\overline{u}^N)
\nonumber \\
=&\sum_{\{ \overline{u}^1, \overline{u}^2,\dots,\overline{u}^N \} \mapsto \{ \overline{v}^1,
\overline{v}^2,\dots,\overline{v}^N \} }
G(\overline{u}^1, \overline{u}^2,\dots,\overline{u}^N|\overline{v}^1,\overline{v}^2,\dots,\overline{v}^N)
T_{NN}(\overline{v}^N) \cdots T_{N2}(\overline{v}^2) T_{N1}(\overline{v}^1), 
\label{generalform}
\end{align}
using only \eqref{fundcommone} and \eqref{fundcommtwo},
and the problem is to determine the coefficients $G(\overline{u}^1, \overline{u}^2,\dots,\overline{u}^N|\overline{v}^1,\overline{v}^2,\dots,\overline{v}^N)$, which we determine
as follows.
First, we observe that the coefficients
$G(\overline{u}^1, \overline{u}^2,\dots,\overline{u}^N|\overline{v}^1,\overline{v}^2,\dots,\overline{v}^N)$
are independent of the representation of the
$L$-operator $T(u)$, that is, the representation of the elements $T_{jk}(u)$,
since we can get the form \eqref{generalform}
in principle by using only \eqref{fundcommone} and \eqref{fundcommtwo}.
We take $T_{jk}(u)$ to be $T_{jk}(u;\overline{v}^1_0,\overline{v}^2_0,\dots,\overline{v}^{N-1}_0)$
where $\{ \overline{v}^1_0,
\overline{v}^2_0,\dots,\overline{v}^N_0 \} $ is one particular choice of $\{ \overline{v}^1,
\overline{v}^2,\dots,\overline{v}^N \} $.
The ordering of variables within each subset $\overline{v}^j$ ($j=1,\dots,N$) can be arbitrary, but we must fix
and use the same ordering for all $L$-operator elements.

We act both hand sides of \eqref{generalform} on
$e_{N^{|\overline{u}^1|+|\overline{u}^2|+\cdots+|\overline{u}^{N-1}|}}$
and take coefficients of
 $e_{1^{|\overline{u}^1|},2^{|\overline{u}^2|},\dots,(N-1)^{|\overline{u}^{N-1}|}}$
of the resulting states.
The left-hand side becomes
$F(\overline{u}^1, \overline{u}^2,\dots,\overline{u}^N|\overline{v}^1_0,\dots,\overline{v}^{N-1}_0)$,
where
\begin{align}
&F(\overline{u}^1, \overline{u}^2,\dots,\overline{u}^N|\overline{v}^1,\dots,\overline{v}^{N-1}) \nonumber \\
:=&
e_{1^{|\overline{u}^1|},2^{|\overline{u}^2|},\dots,(N-1)^{|\overline{u}^{N-1}|}}^*
T_{N1}(\overline{u}^1;\overline{v}^1,\overline{v}^2,\dots,\overline{v}^{N-1}) 
T_{N2}(\overline{u}^2;\overline{v}^1,\overline{v}^2,\dots,\overline{v}^{N-1}) 
\times \cdots 
 \nonumber \\
&\times
T_{NN}(\overline{u}^N;\overline{v}^1,\overline{v}^2,\dots,\overline{v}^{N-1}) 
e_{N^{|\overline{u}^1|+|\overline{u}^2|+\cdots+|\overline{u}^{N-1}|}}.
\label{lefthandsidepartitionfunction}
\end{align}

On the other hand,
one can see the right-hand side can be expressed as
\begin{align}
&\sum_{\{ \overline{u}^1, \overline{u}^2,\dots,\overline{u}^N \} \mapsto \{ \overline{v}^1,
\overline{v}^2,\dots,\overline{v}^N \} } 
G(\overline{u}^1, \overline{u}^2,\dots,\overline{u}^N|\overline{v}^1,\overline{v}^2,\dots,\overline{v}^N)
\nonumber \\
&\times
e_{1^{|\overline{u}^1|},2^{|\overline{u}^2|},\dots,(N-1)^{|\overline{u}^{N-1}|}}^*
T_{NN}(\overline{v}^N;\overline{v}_0^1,\overline{v}_0^2,\dots,\overline{v}_0^{N-1}) 
T_{N N-1}(\overline{v}^{N-1};\overline{v}_0^1,\overline{v}_0^2,\dots,\overline{v}_0^{N-1}) 
\times \cdots 
 \nonumber \\
&\times
T_{N1}(\overline{v}^1;\overline{v}_0^1,\overline{v}_0^2,\dots,\overline{v}_0^{N-1}) 
e_{N^{|\overline{u}^1|+|\overline{u}^2|+\cdots+|\overline{u}^{N-1}|}} \nonumber \\
=&\sum_{\{ \overline{u}^1, \overline{u}^2,\dots,\overline{u}^N \} \mapsto \{ \overline{v}^1,
\overline{v}^2,\dots,\overline{v}^N \} }
G(\overline{u}^1, \overline{u}^2,\dots,\overline{u}^N|\overline{v}^1,\overline{v}^2,\dots,\overline{v}^N)
\nonumber \\
\times&\prod_{j=1}^N
\Bigg\{
e_{
1^{|\overline{u}^1|},
2^{|\overline{u}^{2}|},\dots,
j^{|\overline{u}^{j}|},
N^{|\overline{u}^{j+1}|
+\cdots+|\overline{u}^{N-1}|
}
}^*
T_{Nj}(\overline{v}^j;\overline{v}^1_0,\overline{v}^2_0,\dots,\overline{v}^{N-1}_0)
e_{
1^{|\overline{u}^1|},2^{|\overline{u}^2|},  \dots,
(j-1)^{|\overline{u}^{j-1}|},
N^{|\overline{u}^j|
+\cdots+|\overline{u}^{N-1}|
}
}
\Bigg\}
\nonumber \\
=&\sum_{\{ \overline{u}^1, \overline{u}^2,\dots,\overline{u}^N \} \mapsto \{ \overline{v}^1,
\overline{v}^2,\dots,\overline{v}^N \} }
G(\overline{u}^1, \overline{u}^2,\dots,\overline{u}^N|\overline{v}^1,\overline{v}^2,\dots,\overline{v}^N)
\nonumber \\
\times&\prod_{j=1}^{N-1} \Bigg\{ \prod_{k=1}^{j-1}
(\overline{v}^j-\overline{v}^k_0)
W_j(\overline{v}^j;\overline{v}^j_0)
\prod_{k=j+1}^{N-1} (q \overline{v}^j-q^{-1} \overline{v}^k_0) \Bigg\}
\prod_{k=1}^{N-1}(\overline{v}^N-\overline{v}_0^k), \label{righthandsideelements}
\end{align}
where 
\begin{align}
W_j(\overline{v}^j;\overline{v}^j_0)=e_{j^{|\overline{u}^j|}}^* T_{Nj}(\overline{v}^j;\overline{v}^j_0)e_{N^{|\overline{u}^j|}}.
\end{align}

In some more detail of how to get the right-hand side,
we first compute \\
 $e_{1^{|\overline{u}^1|},2^{|\overline{u}^2|},\dots,(N-1)^{|\overline{u}^{N-1}|}}^*
T_{NN}(\overline{v}^N;\overline{v}^1_0,\overline{v}^2_0,\dots,\overline{v}^{N-1}_0)
$. 
Applying \eqref{icerule}, for $\ell \neq N$,
$R_{Nj}^{k \ell}=0$ unless $j=\ell$, $k=N$.
Using this property,
we can see the action is uniquely determined and
only the matrix elements $R_{Nj}^{Nj}$, $j=1,\dots,N-1$ appear, and we get
\begin{align}
e_{1^{|\overline{u}^1|},2^{|\overline{u}^2|},\dots,(N-1)^{|\overline{u}^{N-1}|}}^*
T_{NN}(\overline{v}^N;\overline{v}^1_0,\overline{v}^2_0,\dots,\overline{v}^{N-1}_0)
=\prod_{k=1}^{N-1}
(\overline{v}^N-\overline{v}^k_0)
e_{1^{|\overline{u}^1|},2^{|\overline{u}^2|},\dots,(N-1)^{|\overline{u}^{N-1}|}}^*.
\end{align}
Next, consider $e_{1^{|\overline{u}^1|},2^{|\overline{u}^2|},\dots,(N-1)^{|\overline{u}^{N-1}|}}^*
T_{N N-1}(\overline{v}^{N-1};\overline{v}^1_0,\overline{v}^2_0,\dots,\overline{v}^{N-1}_0)$.
By a similar observation,
we note the action of $T_{N N-1}(\overline{v}^{N-1};\overline{v}^1_0,\overline{v}^2_0,\dots,\overline{v}^{N-1}_0)$
on the
dual quantum spaces except the last $|\overline{u}|^{N-1}$ spaces are uniquely determined,
and the matrix elements $R_{N-1, j}^{N-1, j}$, $j=1,\dots,N-2$ appear which give the factor
$\prod_{k=1}^{N-2}
(\overline{v}^{N-1}-\overline{v}^k_0)$.
As for the last $|\overline{u}^{N-1}|$ quantum spaces,
we need to consider
$e_{(N-1)^{|\overline{u}^{N-1}|}}^* T_{N,N-1}(\overline{v}^{N-1};\overline{v}^{N-1}_0)$.
By a simple observation,
we conclude
the resulting state must be $e_{N^{|\overline{u}^{N-1}|}}^*$ 
(the resulting state must be a linear combination of
$e_{i_1,i_2,\dots,i_{|\overline{u}^1|+\cdots+|\overline{u}^{N-1}|}}^*$ where $i_j =N-1,N$
since $T_{N-1,N}$ is constructed from the trigonometric $R$-matrix,
and the conservation of colors \eqref{icerule} lead to $i_j=N-1$ for all $j$),
and we get
\begin{align}
e_{(N-1)^{|\overline{u}^{N-1}|}}^* T_{N,N-1}(\overline{v}^{N-1};\overline{v}^{N-1}_0)=
W_{N-1}(\overline{v}^{N-1};\overline{v}^{N-1}_0) e_{N^{|\overline{u}^{N-1}|}}^*.
\end{align}
In total, the action on the dual basis is
\begin{align}
&e_{1^{|\overline{u}^1|},2^{|\overline{u}^2|},\dots,(N-1)^{|\overline{u}^{N-1}|}}^*
T_{N N-1}(\overline{v}^{N-1};\overline{v}^1_0,\overline{v}^2_0,\dots,\overline{v}^{N-1}_0)
\nonumber \\
=&\prod_{k=1}^{N-2}
(\overline{v}^{N-1}-\overline{v}^k_0) W_{N-1}(\overline{v}^{N-1};\overline{v}^{N-1}_0)
e_{
1^{|\overline{u}^1|},
2^{|\overline{u}^{2}|},\dots,
(N-2)^{|\overline{u}^{N-2}|},
N^{|\overline{u}^{N-1}|}
}
.
\end{align}
We can iterate this process to get

\begin{align}
&
e_{1^{|\overline{u}^1|},2^{|\overline{u}^2|},\dots,(N-1)^{|\overline{u}^{N-1}|}}^*
T_{NN}(\overline{v}^N;\overline{v}_0^1,\overline{v}_0^2,\dots,\overline{v}_0^{N-1}) 
T_{N N-1}(\overline{v}^{N-1};\overline{v}_0^1,\overline{v}_0^2,\dots,\overline{v}_0^{N-1}) 
\times \cdots 
 \nonumber \\
&\times
T_{N1}(\overline{v}^1;\overline{v}_0^1,\overline{v}_0^2,\dots,\overline{v}_0^{N-1}) 
e_{N^{|\overline{u}^1|+|\overline{u}^2|+\cdots+|\overline{u}^{N-1}|}} \nonumber \\
&=
\prod_{j=1}^{N-1} \Bigg\{ \prod_{k=1}^{j-1}
(\overline{v}^j-\overline{v}^k_0)
W_j(\overline{v}^j;\overline{v}^j_0)
\prod_{k=j+1}^{N-1} (q \overline{v}^j-q^{-1} \overline{v}^k_0) \Bigg\} \prod_{k=1}^{N-1}(\overline{v}^N-\overline{v}_0^k),
\label{righthandpartitionfunction}
\end{align}
and hence the right-hand side of
\eqref{righthandsideelements} follows.
See 
Figure \ref{fig:righthandpartitionfunction} 
for a graphical description of
\eqref{righthandpartitionfunction}.

\begin{figure}
	\centering
	\includegraphics{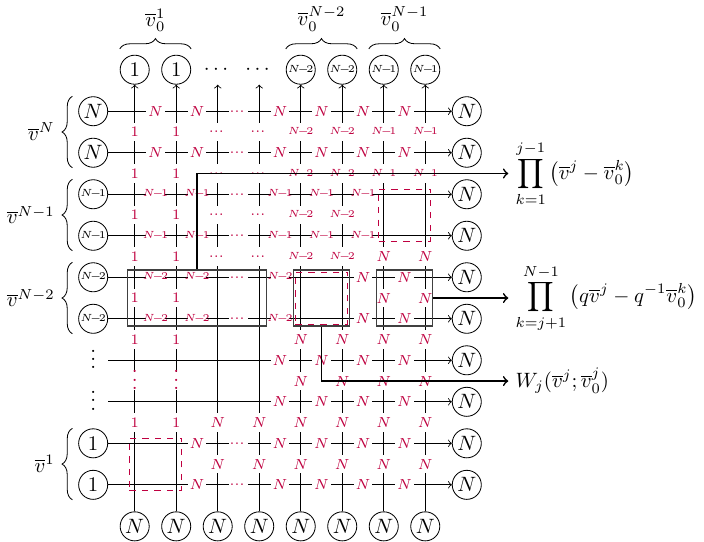}
	\caption{Graphical description of \eqref{righthandpartitionfunction}. The evaluation of the rows corresponding to parameters $\overline{v}^j$ with $j=N-2$ is shown explicitly.} \label{fig:righthandpartitionfunction}
\end{figure}

Next, note that each summand in the right-hand side of
\eqref{righthandsideelements} is a polynomial and
contains the factor
\begin{align}
&\prod_{j=1}^N 
\prod_{k=1}^{j-1} (\overline{v}^j-\overline{v}^k_0),
\end{align}
from which one can see only summands satisfying
$\overline{v}^N \cap \overline{v}^1_0=\overline{v}^N \cap \overline{v}^2_0=\cdots=\overline{v}^N \cap \overline{v}^{N-1}_0=\varnothing$
survive, i.e., $\overline{v}^N$ must be $\overline{v}^N=\overline{v}^N_0$.
Next, one can see nonzero summands should satisfy
$\overline{v}^{N-1} \cap \overline{v}^1_0=\cdots=\overline{v}^{N-1} \cap \overline{v}^{N-2}_0=\varnothing$,
and combining with $\overline{v}^N=\overline{v}^N_0$ one concludes $\overline{v}^{N-1}=\overline{v}^{N-1}_0$.
Repeating this argument, we find only the summand corresponding to $\overline{v}^j=\overline{v}^j_0$, $j=1,\dots,N$ survives.
We can also show
\begin{align}
W_j(\overline{v}^j_0;\overline{v}^j_0)=(q \overline{v}^j_0-q^{-1} \overline{v}^j_0).
\label{domainwallspecialcase}
\end{align}
See Figure~\ref{fig:domainwallspecialcase} for a graphical description to get
\eqref{domainwallspecialcase}.

\begin{figure}
	\centering
	\includegraphics{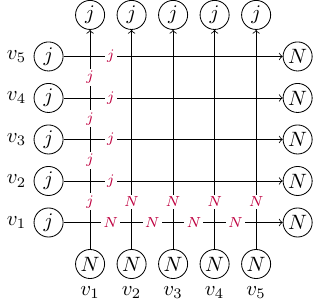}
	\qquad 
	\qquad 
	\includegraphics{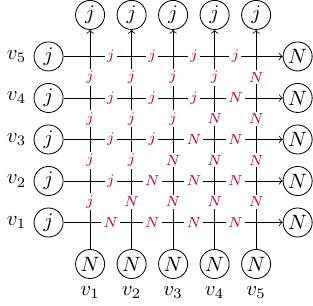}
	\caption{
		Graphical description of \eqref{domainwallspecialcase}. In the left panel, we observe that in lower left-most intersection point the spaces share the same spectral parameter $v_1$, which reduces the possible $R$-matrix elements to just a single choice. This then propagates throughout the first row and column due to color conservation \eqref{icerule}. The same is repeated for the second row and column, and so on, giving just a single term for the whole expression (right panel).
	 } \label{fig:domainwallspecialcase}
\end{figure}

Inserting \eqref{domainwallspecialcase},
the right-hand side of \eqref{righthandsideelements} reduces to the following single term
\begin{align}
G(\overline{u}^1, \overline{u}^2,\dots,\overline{u}^N|\overline{v}^1_0,\overline{v}^2_0,\dots,\overline{v}^N_0)
\prod_{j=1}^{N-1} \Bigg\{ \prod_{k=1}^{j-1}
(\overline{v}^j_0-\overline{v}^k_0)
\prod_{k=j}^{N-1} (q \overline{v}^j_0-q^{-1} \overline{v}^k_0) \Bigg\} \prod_{k=1}^{N-1}(\overline{v}_0^N-\overline{v}_0^k),
\end{align}
and comparing with \eqref{lefthandsidepartitionfunction} we get
\begin{align}
&F(\overline{u}^1, \overline{u}^2,\dots,\overline{u}^N|\overline{v}^1_0,\dots,\overline{v}^{N-1}_0)
\nonumber \\
=&G(\overline{u}^1, \overline{u}^2,\dots,\overline{u}^N|\overline{v}^1_0,\overline{v}^2_0,\dots,\overline{v}^N_0)
\prod_{j=1}^{N-1} \Bigg\{ \prod_{k=1}^{j-1}
(\overline{v}^j_0-\overline{v}^k_0)
\prod_{k=j}^{N-1} (q \overline{v}^j_0-q^{-1} \overline{v}^k_0) \Bigg\} \prod_{k=1}^{N-1}(\overline{v}_0^N-\overline{v}_0^k). \label{rationalcompone}
\end{align}

Next, from the following action
\begin{align}
&T_{NN}(\overline{u}^N;\overline{v}^1,\overline{v}^2,\dots,\overline{v}^{N-1}) 
e_{N^{|\overline{u}^1|+|\overline{u}^2|+\cdots+|\overline{u}^{N-1}|}} \nonumber \\
=&\prod_{j=1}^{N-1} (q \overline{u}^N-q^{-1} \overline{v}^j)
e_{N^{|\overline{u}^1|+|\overline{u}^2|+\cdots+|\overline{u}^{N-1}|}},
\end{align}
which is easy to see,
we get the relation
\begin{align}
&F(\overline{u}^1, \overline{u}^2,\dots,\overline{u}^N|\overline{v}^1,\dots,\overline{v}^{N-1}) \nonumber \\
=&
\prod_{j=1}^{N-1} (q \overline{u}^N-q^{-1} \overline{v}^j)
H(\overline{u}^1, \overline{u}^2,\dots,\overline{u}^{N-1}|\overline{v}^1,\dots,\overline{v}^{N-1}),
\label{rationalcomptwo}
\end{align}
where
$H(\overline{u}^1, \overline{u}^2, \dots,\overline{u}^{N-1}|\overline{v}^1,\dots,\overline{v}^{N-1}) 
$ is defined as \eqref{commcoeffratpart}.

Combining 
\eqref{rationalcompone}
and
\eqref{rationalcomptwo},
we have the following expression for the coefficients of the summands
\begin{align}
&G(\overline{u}^1, \overline{u}^2,\dots,\overline{u}^N|\overline{v}^1_0,\overline{v}^2_0,\dots,\overline{v}^N_0) 
\nonumber \\
=&
\frac{1}{\prod_{1 \le j < k \le N} (\overline{v}^k-\overline{v}^{j})(q \overline{v}^j-q^{-1} \overline{v}^{k-1})}
\prod_{j=1}^{N-1} (q \overline{u}^N-q^{-1} \overline{v}^j)
H(\overline{u}^1, \overline{u}^2,\dots,\overline{u}^{N-1}|\overline{v}^1,\dots,\overline{v}^{N-1}),
\end{align}
in \eqref{generalform}, and the claim follows.

\end{proof}

\clearpage



The second step is
to relate the partition functions on a rectangular grid $H(\overline{u}^1, \overline{u}^2,\dots,\overline{u}^{N-1}|\overline{v}^1,\dots,\overline{v}^{N-1})$ \eqref{commcoeffratpart}
with the trigonometric weight functions \eqref{trigonometric-wavefunction}.
As the derivation is not short, we first
summarize it as Proposition \ref{keyrelationforproof}.
\begin{proposition} \label{keyrelationforproof}
The following holds:
\begin{align}
&H(\overline{u}^1, \overline{u}^2,\dots,\overline{u}^{N-1}|\overline{v}^1,\dots,\overline{v}^{N-1})
\nonumber \\
=&
\frac{1}{
\prod_{j=1}^{N-1} (q \overline{u}^j-q^{-1} \overline{v}^N)
\prod_{\ell=1}^{N-2}
\prod_{j=1}^\ell \prod_{k=1}^{\ell+1} (q \overline{u}^j-q^{-1} \overline{u}^k)}
\nonumber \\
\times& W(
\overline{u}^1,\overline{u}^1 \cup \overline{u}^2,\dots,\overline{u}^1 \cup \cdots \cup \overline{u}^{N-1}
|\overline{v}^1,\overline{v}^2,\dots,\overline{v}^{N-1},\overline{v}^N|1^{|\overline{u}^1|},2^{|\overline{u}^2|},\dots,(N-1)^{|\overline{u}^{N-1}|},N^{|\overline{u}^{N}|} ).
\end{align}

\end{proposition}

\begin{proof}
This follows from combining Lemma \ref{slightlylarger} 
and Proposition \ref{relateoffshellbetherectangular}.
\end{proof}

Let us go to the details of
Lemma \ref{slightlylarger} 
and Proposition \ref{relateoffshellbetherectangular}.



\begin{figure}
	\centering
	\includegraphics{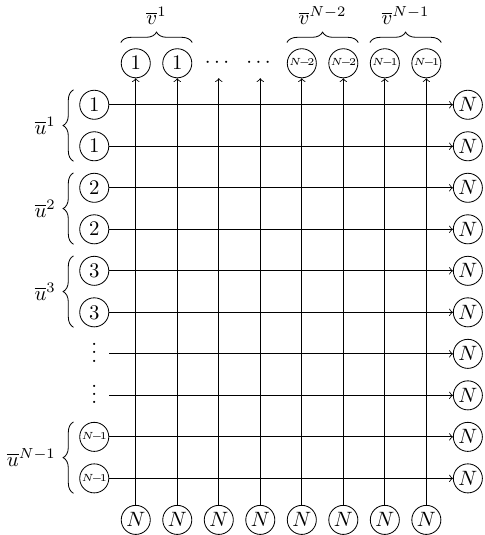}
	\caption{A graphical description of the partition function $H$.} \label{fig:H-function}
\end{figure}

\begin{figure}
	\centering
	\includegraphics{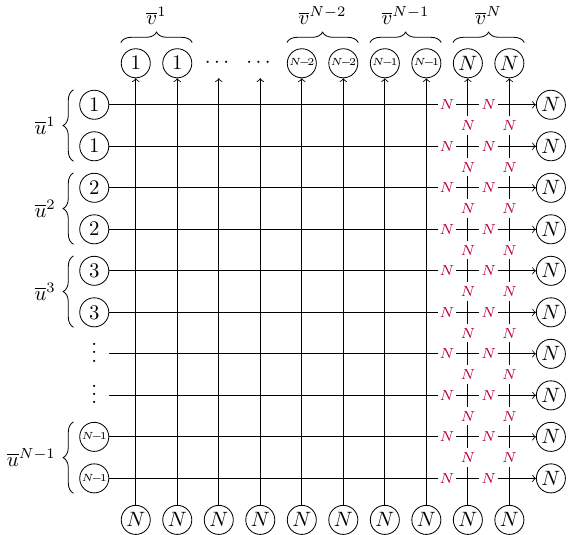}
	\caption{A graphical description of the enlarged partition function $K$. In the rightmost columns, the only possible color is $N$ due to color conservation \eqref{icerule}, which proves Lemma~\ref{slightlylarger}. } \label{fig:K-function}
\end{figure}

We introduce the following partition functions
\begin{align}
&K
(\overline{u}^1,\overline{u}^2,\dots,\overline{u}^{N-1}
|\overline{v}^1,\overline{v}^2,\dots,\overline{v}^{N-1},\overline{v}^N)
\nonumber \\
:=&e_{1^{|\overline{u}^1|},2^{|\overline{u}^2|},\dots,(N-1)^{|\overline{u}^{N-1}|},N^{|\overline{u}^{N}|} }^*
T_{N1}(\overline{u}^1;\overline{v}^1,\overline{v}^2,\dots,\overline{v}^{N-1},\overline{v}^N) 
T_{N2}(\overline{u}^2;\overline{v}^1,\overline{v}^2,\dots,\overline{v}^{N-1},\overline{v}^N) 
\nonumber \\
&\times \cdots \times
T_{NN-1}(\overline{u}^{N-1};\overline{v}^1,\overline{v}^2,\dots,\overline{v}^{N-1},\overline{v}^N) 
e_{N^{|\overline{u}^1|+|\overline{u}^2|+\cdots+|\overline{u}^{N-1}|+|\overline{u}^N|}}.
\end{align}
$K
(\overline{u}^1,\overline{u}^2,\dots,\overline{u}^{N-1}
|\overline{v}^1,\overline{v}^2,\dots,\overline{v}^{N-1},\overline{v}^N)$ is 
a slightly enlarged class of partition functions
(Figure~\ref{fig:K-function}),
and there is a simple relation with
$H(\overline{u}^1, \overline{u}^2,\dots,\overline{u}^{N-1}|\overline{v}^1,\dots,\overline{v}^{N-1})$.

\begin{lemma} \label{slightlylarger}
The following holds:
\begin{align}
&
K
(\overline{u}^1,\overline{u}^2,\dots,\overline{u}^{N-1}
|\overline{v}^1,\overline{v}^2,\dots,\overline{v}^{N-1},\overline{v}^N) \nonumber \\
=&\prod_{j=1}^{N-1} (q \overline{u}^j-q^{-1} \overline{v}^N)
H(\overline{u}^1, \overline{u}^2,\dots,\overline{u}^{N-1}|\overline{v}^1,\dots,\overline{v}^{N-1}).
\label{relationtwopartitionfunctionsgrid}
\end{align}
\end{lemma}

\begin{proof}
This can be checked from a graphical description
of the left-hand side of \eqref{relationtwopartitionfunctionsgrid}.
One can easily see that the configuration on the right part is fixed uniquely,
and the $R$-matrix elements are all $R_{NN}^{NN}$ which give the factor
$\prod_{j=1}^{N-1} (q \overline{u}^j-q^{-1} \overline{v}^N)$ from the frozen part.
The remaining unfrozen part is
$H(\overline{u}^1, \overline{u}^2,\dots,\overline{u}^{N-1}|\overline{v}^1,\dots,\overline{v}^{N-1})$.
This means that the left-hand side of \eqref{relationtwopartitionfunctionsgrid}
can also be expressed by multiplying these factors.
\end{proof}

We next note that the following relation holds.
This is essentially the same as a Proposition
in Borodin-Wheeler \cite[Proposition 7.1.1]{BW}.
We use the standard trigonometric $R$-matrix.

\begin{proposition} \label{relateoffshellbetherectangular}
The following holds:
\begin{align}
&W(
\overline{u}^1,\overline{u}^1 \cup \overline{u}^2,\dots,\overline{u}^1 \cup \cdots \cup \overline{u}^{N-1}
|\overline{v}^1,\overline{v}^2,\dots,\overline{v}^{N-1},\overline{v}^N|1^{|\overline{u}^1|},2^{|\overline{u}^2|},\dots,(N-1)^{|\overline{u}^{N-1}|},N^{|\overline{u}^{N}|} ) \nonumber \\
=&
\prod_{\ell=1}^{N-2}
\prod_{j=1}^\ell \prod_{k=1}^{\ell+1} (q \overline{u}^j-q^{-1} \overline{u}^k)
K(\overline{u}^1, \overline{u}^2,\dots,\overline{u}^{N-1}|\overline{v}^1,\overline{v}^2,\dots,\overline{v}^{N-1},\overline{v}^N).
\end{align}
\end{proposition}

This follows as a specific case 
$\boldsymbol{I}
=(1^{|\overline{u}^1|},2^{|\overline{u}^2|},\dots,(N-1)^{|\overline{u}^{N-1}|},N^{|\overline{u}^{N}|})$
of a more generic relation between partition functions and Theorem \ref{offshellbethetrigonometricweight}.
For the description,
we introduce a partition $J=(J_1,\dots,J_N)$ as in \eqref{partitionJ}.
Let $\overline{w}=\{ w_1,\dots,w_n  \}$
and $\overline{w}_{J_j}=\{ w_k \ | \ k \in J_j  \}$ ($j=1,\dots,N$).
Note $\overline{w}=\overline{w}_{J_1} \cup \cdots \cup \overline{w}_{J_N}$.

\begin{proposition} 
The following holds:
\begin{align}
&\psi(\overline{w}_{J_1},\overline{w}_{J_1} \cup \overline{w}_{J_2},\dots,\overline{w}_{J_1} \cup \cdots \cup \overline{w}_{J_{N-1}}
|\overline{w}|\boldsymbol{I})
\nonumber \\
=&
\prod_{\ell=1}^{N-2}
\prod_{j=1}^\ell \prod_{k=1}^{\ell+1} (q \overline{w}_{J_j}-q^{-1}\overline{w}_{J_k})
K
(\overline{w}_{J_1},\overline{w}_{J_2},\dots,\overline{w}_{J_{N-1}}
|\overline{w}|\boldsymbol{I}), \label{relatingpartitionfunctions}
\end{align}
where
\begin{align}
K
(\overline{w}_{J_1},\overline{w}_{J_2},\dots,\overline{w}_{J_{N-1}}
|\overline{w}|\boldsymbol{I}) 
=&e_{\boldsymbol{I} }^*
T_{N1}(\overline{w}_{J_1};\overline{w}) 
T_{N2}(\overline{w}_{J_2};\overline{w}) 
\times \cdots \times
T_{NN-1}(\overline{w}_{J_{N-1}};\overline{w}) 
e_{N^{n}},
\end{align}
with $e_{\boldsymbol{I} }^*=e^*_{i_1,i_2,\dots,i_n}$ for $I=(i_1,i_2,\dots,i_n)$.

\end{proposition}

Taking $\overline{w}$ to be
$\overline{w}= \overline{v}^1 \cup \overline{v}^2 \cup \dots \cup \overline{v}^{N-1} \cup \overline{v}^N$
and for the case we apply, we can also write as
$\overline{w}
=\overline{u}^1 \cup \overline{u}^2 \cup \dots \cup \overline{u}^N
$, 
and we can take subsets as $\overline{w}_{J_j}=\overline{u}^j$, $j=1,\dots,N-1$.
Setting
$\boldsymbol{I}
=(1^{|\overline{u}^1|},2^{|\overline{u}^2|},\dots,(N-1)^{|\overline{u}^{N-1}|},N^{|\overline{u}^{N}|})$
gives
Proposition \ref{relateoffshellbetherectangular}.

\begin{proof}
$\psi(\overline{w}_{J_1},\overline{w}_{J_1} \cup \overline{w}_{J_2},\dots,\overline{w}_{J_1} \cup \cdots \cup \overline{w}_{J_{N-1}}
|\overline{w}|\boldsymbol{I})
$ is the $\mathfrak{gl}_N$ partition function with parameters specialized so that the `quantum' and `auxiliary' spaces have the same sets of parameters, see Figure~\ref{f:glN-wavefunction}. 




The trick we use is essentially the same with the one given in  \cite[Proposition 7.1.1]{BW}.
	Observe that the ordering of the spaces in each layer is arbitrary -- the input and output coordinate is the same in each layer, so reordering can be accomplished by insertion of an $R$-matrix and application of the Yang-Baxter equation. 
	We thus choose the lexicographic ordering according to the parameter label from left to right and top to bottom in each case. 
	


		The argument then follows by collapsing each layer of the specialized partition function, starting with the lowest layer.
		Indeed, for $1 \leq j \leq N-2$, we consider if we assume that the first 
		$(j-1)$ layers have been collapsed, the $j^\text{th}$ layer will be as in Figure~\ref{f:layerj}.
		This part corresponds to the following vector
\begin{align}
\phi^j&:=T_{j+1,1}(\overline{w}_{J_1}|\overline{w}_{J_1} \cup \overline{w}_{J_2} \cup \cdots \cup \overline{w}_{J_{j+1}})
T_{j+2,1}(\overline{w}_{J_2}|\overline{w}_{J_1} \cup \overline{w}_{J_2} \cup \cdots \cup \overline{w}_{J_{j+1}})
\nonumber \\
&\times \cdots \times  T_{j+1,j}(\overline{w}_{J_j}|\overline{w}_{J_1} \cup \overline{w}_{J_2} \cup \cdots \cup \overline{w}_{J_{j+1}})
e_{j+1}^{|J_1|+|J_2|+\cdots+|J_{j+1}|},
\end{align}
where $|J_j|=|\overline{w}_{J_j}|$.
		
		\begin{figure} 
			\centering
			\includegraphics[width=0.9\textwidth]{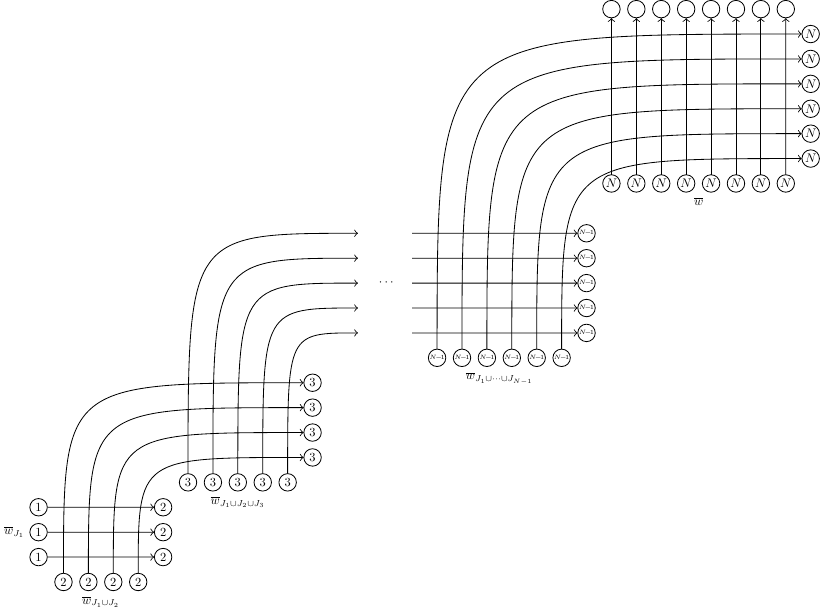}
			\caption{Partition function for the $\mathfrak{gl}_N$ system.}
			\label{f:glN-wavefunction}
		\end{figure}

		\begin{figure}
			\centering
			\includegraphics{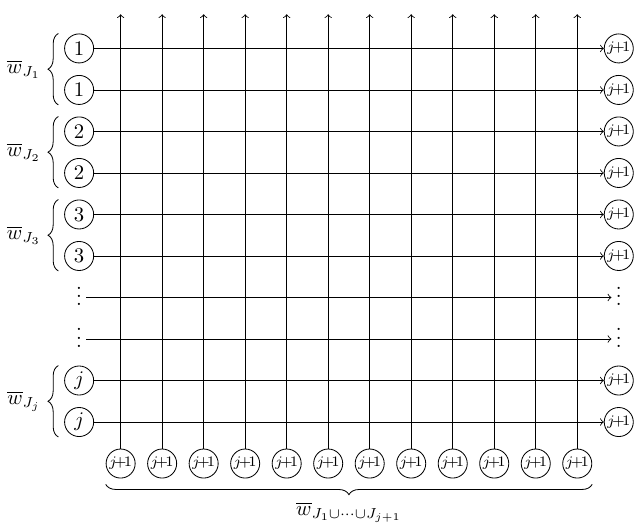}
			\caption{The $j^\text{th}$ layer of the Bethe vector.}
			\label{f:layerj}
		\end{figure}

		Now, with the ordering as above, the intersections on the ``diagonal'' will have the same parameters. 
		Specifically, 
		\[
			R(u,u) = (q-q^{-1})u \left(\sum_{i,j=1}^N  E_{ij} \otimes E_{ji} \right).
		\]

		As such, each can be replaced by a permutation operator multiplied by a constant $(q-q^{-1})u$, giving an overall factor of
		\begin{equation} \label{graphical-arg-factor1}
			\prod_{i \in J_1 \cup \dots \cup J_{j}} \{ (q-q^{-1})w_i \}.
		\end{equation}
		The resulting diagram is given by Figure~\ref{f:layerj-spec}.

		\begin{figure}
			\centering
			\includegraphics{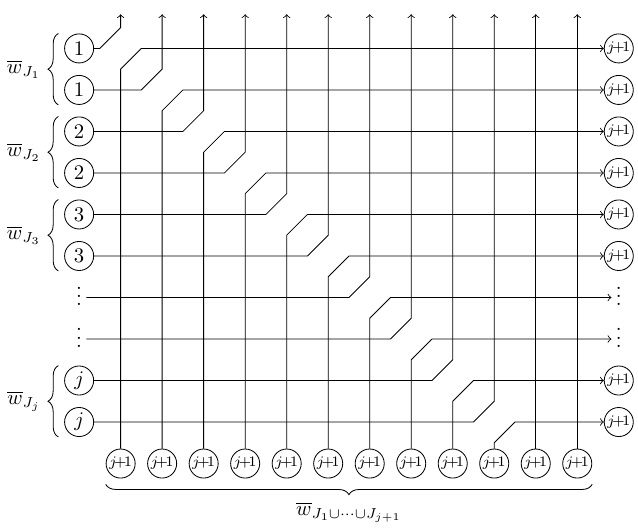}
			\caption{The $j^\text{th}$ layer of the Bethe vector after specialization.}
			\label{f:layerj-spec}
		\end{figure}

		Next, the unitarity relation is applied repeatedly.
		Recall that 
		\begin{equation} \label{trig-unitarity}
			R_{12}(u,v)R_{21}(v,u)
			=
			(qu-q^{-1}v)(qv-q^{-1}u) \mathrm{I} \otimes \mathrm{I}.			
		\end{equation}
		As a result, the Bethe vector is unravelled, resulting in Figure~\ref{f:layerj-unravel}.
		In doing so, each `auxiliary space' interacts with every `quantum space' that is before it in the lexicographic ordering. 
		In other words, the following factor is introduced:
		\[
			\prod_{\substack{k,\ell \in J_1 \cup \cdots \cup J_{j} \\ k<\ell}}(qw_k-q^{-1}w_\ell)(qw_\ell-q^{-1}w_k)
			= \prod_{\substack{k,\ell \in J_1 \cup \cdots \cup J_{j} \\  k\neq \ell}}(qw_k-q^{-1}w_\ell).
		\]
		Combining this with the previous factor \eqref{graphical-arg-factor1}  results in 
		\begin{equation}\label{graphical-arg-factor2}
			\prod_{k,\ell \in J_1 \cup \cdots \cup J_{j} }(qw_k-q^{-1}w_\ell) = q\ol{w}_{J_1 \cup \cdots \cup {J_{j}}}-q^{-1}\, \ol{w}_{J_1 \cup \cdots \cup {J_{j}}}.
		\end{equation}
		\begin{figure}
			\centering
			\includegraphics{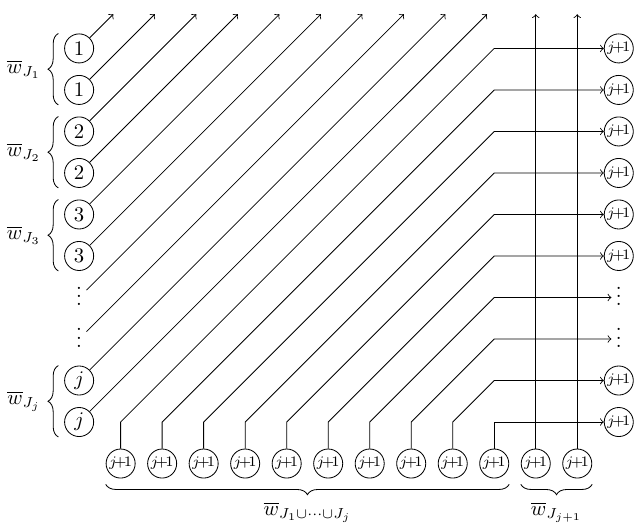}
			\caption{The unravelled Bethe vector.}
			\label{f:layerj-unravel}
		\end{figure}

		Finally, as the remaining coordinates on the right-hand side are all equal to $j+1$, this part of the partition function is also fully determined by the conservation of colors \eqref{icerule}. 
		Each $R$-matrix element that appears is equal to $[R(u,v)]_{j+1,j+1}^{j+1,j+1}$ for some $u,v$, resulting in factor $(q \ol{w}_{J_1 \cup \cdots \cup {J_{j}}}-q^{-1}\,\ol{w}_{J_{j+1}})$.
		Combining this with \eqref{graphical-arg-factor2}, we have an overall factor of 
		\begin{equation*}
			(q\ol{w}_{J_1 \cup \cdots \cup {J_{j}}}-q^{-1}\,\ol{w}_{J_1 \cup \cdots \cup {J_{j+1}}}).
		\end{equation*}
		Graphically, this entire step is described in Figure~\ref{f:graphic-arg-gl5}.
What is shown in the $j^\text{th}$ layer is that $\phi^j$ becomes
\[
			\phi^j =  (q\ol{w}_{J_1 \cup \cdots \cup {J_{j}}}-q^{-1}\,\ol{w}_{J_1 \cup \cdots \cup {J_{j+1}}})
			(e_{1})^{\otimes |J_1|} \otimes (e_{2})^{\otimes |J_2|} \otimes  \cdots \otimes (e_{j+1})^{\otimes |J_{j+1}|}.
		\]

		\begin{figure}[ht]
			\includegraphics[width=\textwidth]{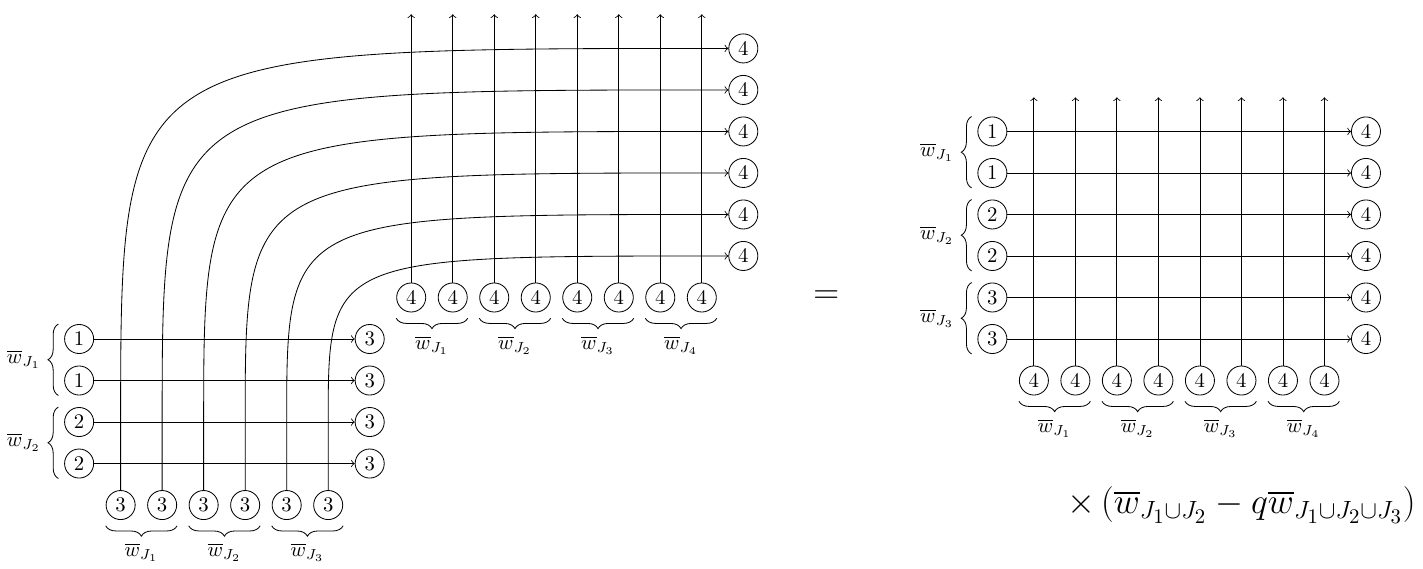}
			\caption{Inductive step of the graphical argument for the specialization of the partition function.}
			\label{f:graphic-arg-gl5}
		\end{figure}

Iterating this process, we have the following factor		
\[
\prod_{j=1}^{N-2} (q\ol{w}_{J_1 \cup \cdots \cup {J_{j}}}-q^{-1}\,\ol{w}_{J_1 \cup \cdots \cup {J_{j+1}}})
=
\prod_{\ell=1}^{N-2}
\prod_{j=1}^\ell \prod_{k=1}^{\ell+1} (q \overline{w}_{J_j}-q^{-1}\overline{w}_{J_k}),
\]
coming from the layers except the last one,
and $(e_{1})^{\otimes |J_1|} \otimes \cdots (e_{N-1})^{\otimes |J_{N-1}|}$
which is the basis vector for $\phi^{N-2}$ corresponds to the coloring
of the left side of the last layer,
and we have the right panel of Figure~\ref{f:graphic-arg-gl5}.
Namely, the last layer becomes the partition function
$
K
(\overline{w}_{J_1},\overline{w}_{J_2},\dots,\overline{w}_{J_{N-1}}
|\overline{w}|\boldsymbol{I})
$.
Hence we get \eqref{relatingpartitionfunctions}.
		
		\end{proof}
		
{\it Proof of Theorem \ref{thmmultiplecommutationrelations}. } \\
Combining Proposition \ref{keypropositionforproof}
and Proposition \ref{keyrelationforproof} gives Theorem \ref{thmmultiplecommutationrelations}.
\hfill$\square$ \\

As a final remark of this section, let us present another form
for the case $N=2$, recovering another presentation
using the Izergin-Korepin determinants \cite[(7.62), (7.63)]{Slavnovbook}.
The case $N=2$ of Proposition \ref{keypropositionforproof} is

\begin{align}
&T_{21}(\overline{u}^1) T_{22}(\overline{u}^2) 
=\sum_{\{ \overline{u}^1, \overline{u}^2 \} \mapsto \{ \overline{v}^1,
\overline{v}^2 \} } 
\frac{
q \overline{u}^2-q^{-1} \overline{v}^1
}{ (\overline{v}^2-\overline{v}^{1})(q \overline{v}^1-q^{-1} \overline{v}^{1})}
H(\overline{u}^1|\overline{v}^1)
T_{22}(\overline{v}^2)  T_{21}(\overline{v}^1), 
\end{align}
where
\begin{align}
&H(\overline{u}^1|\overline{v}^1) 
=
e_{1^{|\overline{u}^1|}}^*
T_{21}(\overline{u}^1;\overline{v}^1) 
e_{2^{|\overline{u}^1|}}.
\end{align}

By
\eqref{domainwallpartitionIKdet}, $H(\overline{u}^1|\overline{v}^1)$
can be expressed as the Izergin-Korepin determinant $K(\overline{u}^1|\overline{v}^1)$,
and we get
\begin{align}
&T_{21}(\overline{u}^1) T_{22}(\overline{u}^2) 
=\sum_{\{ \overline{u}^1, \overline{u}^2 \} \mapsto \{ \overline{v}^1,
\overline{v}^2 \} } 
\frac{
q \overline{u}^2-q^{-1} \overline{v}^1
}{ (\overline{v}^2-\overline{v}^{1})(q \overline{v}^1-q^{-1} \overline{v}^{1})}
K(\overline{u}^1|\overline{v}^1)
T_{22}(\overline{v}^2)  T_{21}(\overline{v}^1); 
\end{align}
setting
$i=2$, $j=2$, $k=1$, $\overline{u}^1=\{ u \}$,
$\overline{u}^2=\{ v \}$ recovers the basic commutation relation.

\section{A construction the Gelfand-Tsetlin basis for the vector representation}
The relation \eqref{relatingpartitionfunctions} implies that
$T_{N1}(\overline{w}_{J_1};\overline{w}) 
T_{N2}(\overline{w}_{J_2};\overline{w}) 
\ \cdots 
T_{NN-1}(\overline{w}_{J_{N-1}};\overline{w}) 
e_{N^{n}}
$ for all $J$ give rise to a construction of the Gelfand-Tsetlin basis
for the case of the tensor product of vector representations,
a basis which simultaneously diagonalizes the Gelfand-Tsetlin subalgebra.
There are two known main constructions of the Gelfand-Tsetlin basis,
one due to Nazarov-Tarasov \cite{NT} and another due to Molev \cite{Molev}.
We briefly review Molev's construction in the Appendix, which shows that
$T_{21}(w_{J_1};\overline{w}) T_{32}(w_{J_1} \cup w_{J_2};\overline{w}) \cdots T_{N,N-1}(w_{J_1} \cup \cdots \cup w_{J_{N-1}};\overline{w}) e_{N^n}
$ for all $J$ give rise to a construction of the Gelfand-Tsetlin basis.
However, it seems not easy to show directly that
$T_{N1}(\overline{w}_{J_1};\overline{w}) 
T_{N2}(\overline{w}_{J_2};\overline{w}) 
\cdots 
T_{NN-1}(\overline{w}_{J_{N-1}};\overline{w}) 
e_{N^{n}}
$ for $N>2$ diagonalizes the Gelfand-Tsetlin subalgebra.
We take an indirect approach and derive the following relations
between \\
$T_{N1}(\overline{w}_{J_1};\overline{w}) 
T_{N2}(\overline{w}_{J_2};\overline{w}) $
$\cdots 
T_{NN-1}(\overline{w}_{J_{N-1}};\overline{w}) 
e_{N^{n}}
$  and $T_{21}(w_{J_1};\overline{w}) T_{32}(w_{J_1} \cup w_{J_2};\overline{w}) \cdots T_{N,N-1}(w_{J_1} \cup \cdots \cup w_{J_{N-1}};\overline{w}) e_{N^n}
$.
\begin{proposition} 
The following relation holds:
\begin{align}
	&T_{21}(\ol{w}_{J_1};\ol{w})
	T_{32}(\ol{w}_{J_{1}}\cup\ol{w}_{J_2};\ol{w})
	\cdots
	T_{N,N-1}(\ol{w}_{J_{1}}\cup \cdots \cup\ol{w}_{J_{N-1}};\ol{w})e_{N^n}
	\nonumber \\
	=&
	\prod_{j=2}^{N-1}
	\left(\ol{w}_{J_{1}}\cup \cdots \cup \ol{w}_{J_{j-1}} -\ol{w}_{J_{j+1}}\cup \cdots \cup \ol{w}_{J_{N}}\right) 
	\nonumber \\
	&\times
	\prod_{1 \le j < k \le N-1 } \left(q \ol{w}_{J_{j}}-q^{-1}\ol{w}_{J_{k}}\right)^{N-k}
		\left(q \ol{w}_{J_{k}}-q^{-1}\ol{w}_{J_{j}}\right)^{N-k-1}
		\nonumber \\
		&\times
	\prod_{j=1}^{N-2} \left(q \ol{w}_{J_{j}}-q^{-1} \ol{w}_{J_{j}}\right)^{N-j-1}
	T_{N1}(\ol{w}_{J_1};\ol{w}) T_{N2}(\ol{w}_{J_2};\ol{w})
	\cdots 
	T_{N,N-1}(\ol{w}_{J_{N-1}};\ol{w})
	e_{N^n}. \label{relationGZ}
\end{align}

\end{proposition}
The relation
\eqref{relationGZ}, together with Proposition \ref{Molevdiagonalization} in the Appendix, implies the following.
\begin{proposition}
$\xi_J:=T_{N1}(\overline{w}_{J_1};\overline{w}) 
T_{N2}(\overline{w}_{J_2};\overline{w}) 
\cdots 
T_{NN-1}(\overline{w}_{J_{N-1}};\overline{w}) 
e_{N^{n}}$
diagonalizes the quantum determinants \eqref{quantumdeterminant}:
\begin{align}
\mathrm{qdet} T^{(j)}(u) \cdot \xi_J
=\prod_{k=1}^j \lambda_{jk}^J (q^{2k-2}u) \xi_J.
\end{align}
\end{proposition}

We use recent results
for the universal nested Bethe vectors by Pakuliak-Ragoucy-Slavnov \cite{PRS}
which uses the following version of the $R$-matrix
\begin{align}
	\tilde{R}(u,v)=&(qu-q^{-1}v) \sum_{1 \leq i \leq N} E_{ii}\otimes E_{ii}
	+ \sum_{1 \leq i < j \leq N} (u-v)(E_{ii}\otimes E_{jj} + E_{jj} \otimes E_{ii}) 
\nonumber	\\
&+ \sum_{1 \leq i < j \leq N} ((q-q^{-1})u E_{ij} \otimes E_{ji} + (q-q^{-1})v E_{ji} \otimes E_{ij}).
\label{trigRvertwo}
\end{align}
Note the $R$-matrix elements
for \eqref{trigRverone} and \eqref{trigRvertwo}
 are related by
$[\tilde{R}(u,v)]_{ij}^{k \ell}=[R(u,v)]_{N+1-i,N+1-j}^{N+1-k,N+1-\ell}$.
We denote the $L$-operator and its elements constructed using this $R$-matrix by $\tilde{T}(u)$,
$\tilde{T}_{ij}(u)$.
In this version,  relations are given by the following.
\begin{proposition} \label{showrelationprop}
The following relation holds:
\begin{align}
	&\tilde{T}_{N-1,N}(\ol{w}_{I_N}|\ol{w})
	\tilde{T}_{N-2,N-1}(\ol{w}_{I_{N-1}}\cup\ol{w}_{I_N}|\ol{w})
	\cdots
	\tilde{T}_{1,2}(\ol{w}_{I_{2}}\cup \cdots \cup\ol{w}_{I_N}|\ol{w})e_{1^n}
	\nonumber
	\\
	=&
	\prod_{j=2}^{N-1}
	\left(\ol{w}_{I_{j+1}}\cup \cdots \cup \ol{w}_{I_{N}} -\ol{w}_{I_{1}}\cup \cdots \cup \ol{w}_{I_{j-1}}\right)
	\nonumber
	\\
&	\times 
	\prod_{1 \le j<k \le N-1} \left(q \ol{w}_{I_{k+1}}-q^{-1}\ol{w}_{I_{j+1}}\right)^j
	\left(q \ol{w}_{I_{j+1}}-q^{-1}\ol{w}_{I_{k+1}}\right)^{j-1}
	\nonumber
	\\
&	\times 
	\prod_{j=2}^{N-1} \left(q \ol{w}_{I_{j+1}}-q^{-1} \ol{w}_{I_{j+1}}\right)^{j-1}
	\tilde{T}_{1N}(\ol{w}_{I_N}|\ol{w})
	\cdots 
	\tilde{T}_{12}(\ol{w}_{I_2}|\ol{w})
	e_{1^n}. \label{showrelationGZ}
\end{align}

\end{proposition}

Since the $R$-matrix elements are related by
$[\tilde{R}(u,v)]_{ij}^{k \ell}=[R(u,v)]_{N+1-i,N+1-j}^{N+1-k,N+1-\ell}$,
\eqref{relationGZ} follows from
\eqref{showrelationGZ} by replacing $\tilde{T}_{ij}$ by $T_{N+1-i,N+1-j}$,
$\overline{w}_{I_j}$ by $\overline{w}_{J_{N+1-j}}$ and $e_{1^n}$ by $e_{N^n}$.

We use two different expressions for the same object
(nested Bethe vectors) in \cite{PRS} for $U_q(\widehat{\mathfrak{gl}_{N}})$.
Here we present the minimal necessary results. We refer to \cite{PRS} for more details.  
For the description, we introduce symbols for the following rational function
\[
	f(u,v) = \frac{qu-q^{-1}v}{u-v},
\]
	and the following determinant
	\[
		K_n^{\mathrm{rat}}(\ol{u}|\ol{v}) := 
		\frac{\prod_{i,j =1}^n(qu_i-q^{-1}v_j)}{\prod_{1\leq i <j \leq n}(u_i-u_j)(v_j-v_i)}
		\det_{1 \le i,j \le n} \left[\frac{q-q^{-1}}{(qu_i-q^{-1}v_j)(u_i-v_j)}\right].
	\]
	The `left' and `right' versions of the Izergin-Korepin determinant is defined as
	\[
		K^{(l),\mathrm{rat}}_n(\ol{u}|\ol{v}) := 
		K_n(\ol{u}|\ol{v}) \prod_{j=1}^n u_j;
		\qquad 
		K^{(r),\mathrm{rat}}_n(\ol{u}|\ol{v}) := 
		K_n(\ol{u}|\ol{v}) \prod_{j=1}^n v_j.
	\]

	These correspond to determinant representations of the domain wall boundary partition functions
	of the six-vertex model using the rational version of the $R$-matrix
	\[  \tilde{R}^{\mathrm{rat}}(u,v)=(u-v)^{-1} \tilde{R}(u,v).
	\]
	See Figure~\ref{f:IK-det-lr}.
	Then, we define the polynomial versions of these quantities as follows:
	\begin{align}
		K^{(l)}_n(\ol{u}|\ol{v}) &:=  \prod_{i,j=1}^n (u_i-v_j) K^{(l),\mathrm{rat}}_n(\ol{u}|\ol{v}) \nonumber \\
		&=\frac{1}{\prod_{1\leq i <j \leq n}(u_i-u_j)(v_j-v_i)}
		\det_{1 \le i,j \le n} \left[
		(q-q^{-1})u_i
		\prod_{\substack{k=1 \\ k \neq j}}^n
		(qu_i-q^{-1}v_k)(u_i-v_k)
		\right], \label{leftDW}  \\
		K^{(r)}_n(\ol{u}|\ol{v}) &:=  \prod_{i,j=1}^n (u_i-v_j) K^{(r),\mathrm{rat}}_n(\ol{u}|\ol{v}) \nonumber \\
		&=\frac{1}{\prod_{1\leq i <j \leq n}(u_i-u_j)(v_j-v_i)}
		\det_{1 \le i,j \le n} \left[
		(q-q^{-1})v_j
		\prod_{\substack{k=1 \\ k \neq j}}^n
		(qu_i-q^{-1}v_k)(u_i-v_k)
		\right]. \label{rightDW}
	\end{align}
	These correspond to the domain wall partition functions with using
	$\tilde{R}(u,v)$ instead of $\tilde{R}^{\mathrm{rat}}(u,v)$.
	See also
	\eqref{introductionDW}.
	
	We use the following properties in the next two subsections
	\begin{align}
	K^{(l)}_n(\ol{u}|\ol{u})=K^{(r)}_n(\ol{u}|\ol{u})=q \overline{u}-q^{-1} \overline{u},
	\label{factorizationIZrl}
	\end{align}
	which can be understood from graphical descriptions of the domain wall boundary partition functions,
	or by taking the limit $u_j \to v_j$, $j=1,\dots,n$ of \eqref{leftDW}, \eqref{rightDW}.

	\tikzset{
        coord/.style={draw, circle, inner sep=0pt, minimum size=14pt},
        }

	\begin{figure}[h]
		\includegraphics{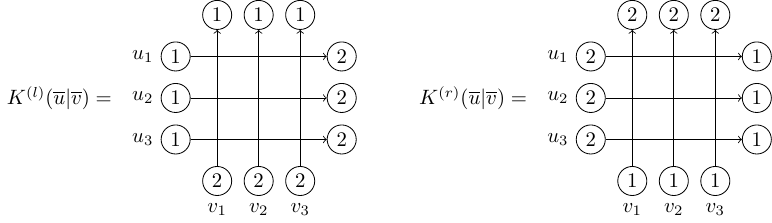}
		\caption{Graphical description of the `left' and `right' Izergin-Korepin determinants.}
		\label{f:IK-det-lr}
	\end{figure}
	
	The description of the universal Bethe vector in \cite{PRS}
	uses the $L$-operator using the rational $\tilde{R}^{\mathrm{rat}}(u,v)$ version,
	which we denote in this paper by $\tilde{T}^{\mathrm{rat}}(u)$.


We introduce the set of spectral variables $\bar{t}^{\,k}$, $k=1,\dots,N-1$
and decomposition into disjoint union of sets
\[
\bar{t}^{\,k}
=\;
\bigcup_{i=1}^{k}
\;\;\bigcup_{j=k}^{N-1}
\bar{t}_{i,j}^{\,k},
\qquad
1 \le i \le k \le j \le N-1,
\]
satisfying
$
|\bar{t}^{\,k}_{i,j}| \;=\; | \bar{t}^{\,k'}_{i,j}|,
\ \forall\, k,k' .$

We introduce two different types of ordering of indices,
\[
	i,j \prec i',j' \iff i < i' \text{ or } i=i', j<j',
\]
\[
	i,j \prec^t i',j' \iff j < j' \text{ or } j=j', i<i'.
\]

Two expressions in \cite[Proposition 3.1]{PRS} for the universal Bethe vector
lead to the following identity.
\begin{theorem}\cite[Proposition 3.1]{PRS} \label{bethevectorequivalence}
The following relation holds:
\begin{align}
B(\bar{t}^{\,1},\bar{t}^{\,2},\dots,\bar{t}^{\,N-1})=\widehat{B}(\bar{t}^{\,1},\bar{t}^{\,2},\dots,\bar{t}^{\,N-1}),
\end{align}
where
\begin{align}
			B(\bar{t}^{\,1},\bar{t}^{\,2},\dots,\bar{t}^{\,N-1})=&\sum_{\mathrm{part}}
			\prod_{k=1}^{N-1} \prod_{i,j \prec i',j'}
			f(\ol{t}^k_{i',j'},\ol{t}^k_{i,j})
			\prod_{k=2}^{N-1} \left(\prod_{i,j \prec i',j'} f(\ol{t}^k_{i,j},\ol{t}^{k-1}_{i',j'})
			\prod_{i<j} K^{(l),\mathrm{rat}}(\ol{t}^k_{i,j}|\ol{t}^{k-1}_{i,j})\right)
			\nonumber
			\\
&			\times
			\prod^{\longrightarrow}_{1 \leq k \leq N-1} 
			\left(
				\prod^{\longleftarrow}_{N \geq j > k} \tilde{T}^{\mathrm{rat}}_{kj}(\ol{t}^k_{k,j-1})
			\right)
			\prod_{k=2}^{N-1}\prod_{i,j\prec k,k} \tilde{T}^{\mathrm{rat}}_{k,k}(\ol{t}^k_{i,j}),
			\label{firstuniversalBethe}
		\end{align}

\begin{align}
	\widehat{B}(\bar{t}^{\,1},\bar{t}^{\,2},\dots,\bar{t}^{\,N-1})=&\sum_{\mathrm{part}}
	\prod_{k=1}^{N-1} \prod_{i,j \prec i',j'}
	f(\ol{t}^k_{i',j'},\ol{t}^k_{i,j})
	\prod_{k=2}^{N-1} \left(\prod_{i,j \prec^t i',j'} f(\ol{t}^k_{i,j},\ol{t}^{k-1}_{i',j'})
	\prod_{i<j} K^{(r),\mathrm{rat}}(\ol{t}^k_{i,j}|\ol{t}^{k-1}_{i,j})\right)
	\nonumber
	\\
	&\times
	\prod^{\longleftarrow}_{N-1 \geq k \geq 1} 
	\left(
		\prod^{\longrightarrow}_{1 \leq j \leq k} \tilde{T}^{\mathrm{rat}}_{j,k+1}(\ol{t}^k_{j,k})
	\right)
	\prod_{k=1}^{N-2}\prod_{k,k\prec^t i,j} \tilde{T}^{\mathrm{rat}}_{k+1,k+1}(\ol{t}^k_{i,j}).
	\label{seconduniversalBethe}
\end{align}

Here, we take sum over all partitions of $\overline{t}^k$ for each $k \ (k=1,\dots,N-1)$ into subsets
$\bar{t}_{i,j}^{\,k}, 1 \le i \le k \le j \le N-1$ satisfying
$
|\bar{t}^{\,k}_{i,j}| \;=\; | \bar{t}^{\,k'}_{i,j}|,
\ \forall\, k,k' .$
The ordered product symbol, denoted as
 $\displaystyle \prod^{\longrightarrow}$, indicates a product where terms are multiplied in ascending index order. Conversely, the symbol $\displaystyle \prod^{\longleftarrow}$ represents a product where terms are multiplied in descending index order.

\end{theorem}

{\it Proof of Proposition \ref{showrelationprop}.
}

We multiply
\eqref{firstuniversalBethe} and \eqref{seconduniversalBethe}
by the same overall factor, take the same vector representation and
act on the highest weight vector,
which we denote by $\Psi$ and $\widetilde{\Psi}$. Explicitly,
\begin{align}
		\Psi &= (\ol{t}^1-\ol{w})\prod_{\ell=2}^{N-1} \left(\ol{t}^{\ell}-\ol{t}^{\ell-1}\right)
		B(\bar{t}^{\,1},\bar{t}^{\,2},\dots,\bar{t}^{\,N-1}) e_{1^n}, \\
		\widetilde{\Psi} &= (\ol{t}^1-\ol{w})\prod_{\ell=2}^{N-1} \left(\ol{t}^{\ell}-\ol{t}^{\ell-1}\right)
		\widehat{B}(\bar{t}^{\,1},\bar{t}^{\,2},\dots,\bar{t}^{\,N-1}) e_{1^n}.
		\end{align}

We further specialize the variables $\ol{t}^j$ in the same way
\begin{align*}
		\ol{t}^j &= \ol{w}_{I_{j+1}} \cup \ol{w}_{I_{j+2}} \cup \cdots \cup \ol{w}_{I_N},
		\ \ \ j=1,\dots,N-1,
	\end{align*}
which yields
\eqref{psi-expr1-final}
and
\eqref{psi-expr2-final} respectively.
Together with Theorem \ref{bethevectorequivalence},
we get \eqref{showrelationGZ}. \hfill$\square$ \\

In subsections \ref{sec:firstspec} and \ref{sec:secondspec},
we provide the details of specializing $\Psi$ and $\widetilde{\Psi}$
to get \eqref{psi-expr1-final}
and
\eqref{psi-expr2-final}.

\subsection{Relation to the trace formula}

For the $U_q(\widehat{{\mathfrak{gl}}}_3)$ case, the Tarasov-Varchenko trace formula for the Bethe vector \cite{TVsigma} is given by 
\begin{multline}
	\mathrm{tr}_{k_1,\dots,k_a,n_1,\dots,n_b}\bigg(T_{k_1}(u_1)\cdots T_{k_a}(u_a) T_{n_1}(v_1)\cdots T_{n_b}(v_b) \prod_{i=1}^b\prod_{j=1}^a R_{n_i,k_j}(v_i,u_j)
	\\
	\times E_{k_1}^{23} \cdots E_{k_a}^{23} E_{n_1}^{12}\dots E_{n_b}^{12}\bigg) \ket{\Omega},
\end{multline}
where we have used notations similar to \cite{Slavnov}. 
Here $\ket{\Omega}$ is the vacuum state of a representation of $U_q(\widehat{{\mathfrak{gl}}}_3)$, $E_{a}^{ij}$ are the standard matrix units acting on space $V_a$. 
See Figure~\ref{fig:traceformula} for a diagrammatic representation. 

\begin{figure}
	\centering
	\includegraphics{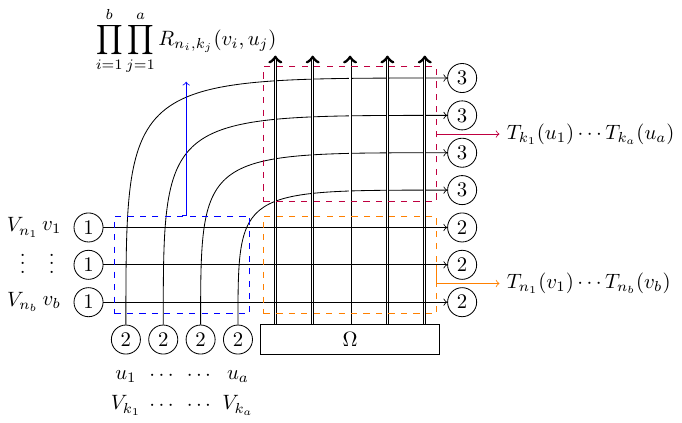}
	\caption{A pictorial representation of the trace formula for $U_q(\widehat{{\mathfrak{gl}}}_3)$.}
	\label{fig:traceformula}
\end{figure}

In the particular case where the quantum space is $\underbrace{V \otimes \cdots \otimes V}_{p}$ with associated spectral parameters $\overline{w} = \{w_1, \dots, w_p\}$ and $\ket{\Omega} = e_{3^p}$, many of the $L$ operators have only one nonzero component due to \eqref{icerule}, as depicted in Figure~\ref{fig:traceformulafreeze}.
\begin{figure}
	\centering
	\includegraphics{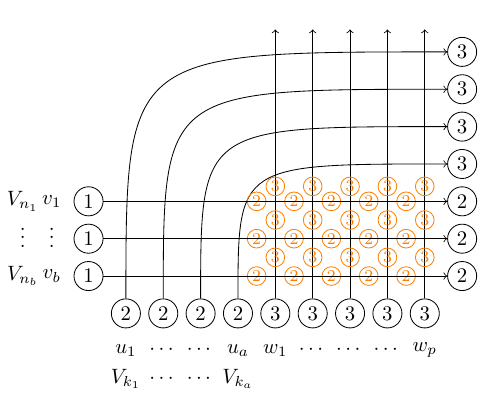}
	\caption{The lower-right grid becomes fully determined using \eqref{icerule}.}
	\label{fig:traceformulafreeze}
\end{figure}
As a result, the trace formula is equivalent to Figure~\ref{fig:traceformularesult}.
\begin{figure}
	\centering
	\includegraphics{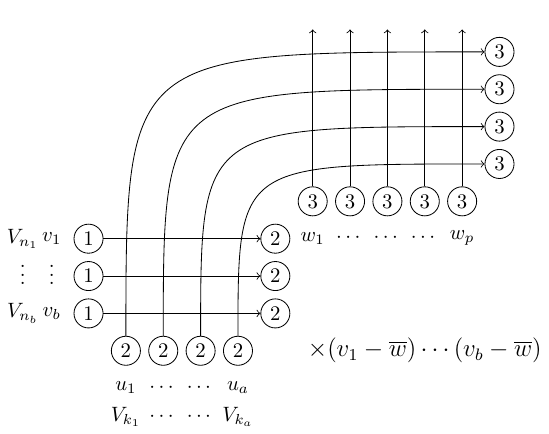}
	\caption{The result of applying \eqref{icerule} to Figure~\ref{fig:traceformulafreeze} }
	\label{fig:traceformularesult}
\end{figure}
Specifying the top boundary of this figure essentially means taking a component of the
trace formula of the nested Bethe vector
and corresponds to $N=3$ case of Figure \ref{fig:partitionfn} in the previous section.
As in the previous section, further specializing $\{v_1, \dots v_b\} \to \overline{w}_{J_1}$ and $\{u_1, \dots u_a\} \to  \overline{w}_{J_1} \cup \overline{w}_{J_2}$
gives
\[
	T_{31}(\overline{w}_{J_1};\overline{w})
	T_{32}(\overline{w}_{J_2};\overline{w})
	e_{3^p},
\]
which agrees with the right-hand side of \eqref{relationGZ} in the case $N=3$ up to an overall factor.

\subsection{First specialization} \label{sec:firstspec}


	We take the tensor product of vector representation for the universal Bethe vector
$B(\bar{t}^{\,1},\bar{t}^{\,2},\dots,\bar{t}^{\,N-1})$ \eqref{firstuniversalBethe},
i.e. take the $L$-operator to be
$
\tilde{T}^{\mathrm{rat}}(u)=
\tilde{T}^{\mathrm{rat}}(u;\overline{w})=\tilde{R}^{\mathrm{rat}}_{0n}(u,w_n) \cdots \tilde{R}_{02}^{\mathrm{rat}}(u,w_2) \tilde{R}_{01}^{\mathrm{rat}}(u,w_1)
$. We multiply by the overall factor $(\ol{t}^1-\ol{w})\prod_{\ell=2}^{N-1} \left(\ol{t}^{\ell}-\ol{t}^{\ell-1}\right)$
and act on the highest weight vector $e_{1^n}$.
	
		\begin{align}
		\Psi &:= (\ol{t}^1-\ol{w})\prod_{\ell=2}^{N-1} \left(\ol{t}^{\ell}-\ol{t}^{\ell-1}\right)
		B(\bar{t}^{\,1},\bar{t}^{\,2},\dots,\bar{t}^{\,N-1}) e_{1^n}
		\nonumber \\
		&=
		\left(\ol{t}^1-\ol{w}\right)\prod_{\ell=2}^{N-1} \left(\ol{t}^{\ell}-\ol{t}^{\ell-1}\right)
		\sum_{\mathrm{part}}
			\prod_{k=1}^{N-1} \prod_{i,j \prec i',j'}
			\frac{q\ol{t}^k_{i',j'}-q^{-1}\ol{t}^k_{i,j}}{\ol{t}^k_{i',j'}-\ol{t}^k_{i,j}}
			\nonumber
			\\
			&\times
			\prod_{k=2}^{N-1} \left(\prod_{i,j \prec i',j'} \frac{q\ol{t}^k_{i,j}-q^{-1}\ol{t}^{k-1}_{i',j'}}{\ol{t}^k_{i,j}-\ol{t}^{k-1}_{i',j'}}
			\prod_{i<j} K^{(l),\mathrm{rat}}(\ol{t}^k_{i,j}|\ol{t}^{k-1}_{i,j})\right)
			\nonumber
			\\
			&\times
			\prod^{\longrightarrow}_{1 \leq k \leq N-1} 
				\prod^{\longleftarrow}_{N \geq j > k} \tilde{T}^{\mathrm{rat}}_{kj}(\ol{t}^k_{k,j-1};\ol{w})
			\prod_{k=2}^{N-1}\prod_{i,j\prec k,k} \tilde{T}^{\mathrm{rat}}_{k,k}(\ol{t}^k_{i,j};\ol{w})
			e_{1^n}.
	\end{align}

Note that for $j>k\geq 2$, 
	\[
		\tilde{T}^{\mathrm{rat}}_{kk}(u;\overline{w}) e_{1^n} = e_{1^n} \qquad \text{and} \qquad  \tilde{T}^{\mathrm{rat}}_{kj}(u;\overline{w}) e_{1^n} = 0.
	\]
	This implies that $\ol{t}^k_{k,j-1} = \varnothing$ for $j >k \geq 2$, for all remaining partitions. 
	This effect can be visualized as in Figure~\ref{f:emptyset}.

	\begin{figure}[ht]
		\begin{tikzpicture}[every pin/.style={red, inner sep=1pt},pin distance=5pt,every pin edge/.style={red, double}]
			\foreach \y in {2,3}
			{
				\foreach \x in {\y,...,4}
				{
					\ifthenelse{\x<4}{
						\node [anchor=center] (t\y\x) at (\x,6-\y) {$\ol{t}^\y_{2,\x}$};
						\node [anchor=center] at (\x+0.5,6-\y) {$\cup$};
					}{
						\node [anchor=center]at (\x,6-\y) {$\cdots$};
						\node [anchor=center] at (\x+0.5,6-\y) {$\cup$};
					}
				}
				\node[anchor=west] (tj\y) at (4.7, 6-\y) {$\ol{t}^{\y}_{2,N-1}$};
			}
			\node [anchor=center]at (4,2.1) {$\ddots$};
			\node [anchor=west] (tjj) at (4.7,1) {$\ol{t}^{N-1}_{2,N-1}$};	

			\draw[thick] (t23) -- (t33);
			\draw[thick] (tj2) -- (tj3);

			\draw[thick] (tj3) -- node [midway,fill=white,inner sep=7pt] {$\rvdots$} (tjj);

			\node at (7,2.5){$\dots$} ;
			\tikzset{shift={(9,2)}}

			\node [anchor=center] (tn2n2) at (0,1) {$\ol{t}^{N-2}_{N-2,N-2}$};
			\node [anchor=center] (tn2n1) at (2,1) {$\ol{t}^{N-2}_{N-2,N-1}$};
			\node [anchor=center] at (1,1) {$\cup$};
			\node [anchor=center] (tn1n1) at (2,0) {$\ol{t}^{N-1}_{N-2,N-1}$};
			\draw [thick] (tn2n1) -- (tn1n1);

			\node [anchor=center] (tnn) at (5,0.5) {$\ol{t}^{N-1}_{N-1,N-1}$};

			\foreach \n in {t22,t23,tj2,tn2n2,tn2n1,tnn}{
				\node also [pin=45:$\varnothing$] (\n);
			}

		\end{tikzpicture}
		\caption{Some sets of parameters become empty sets.}
		\label{f:emptyset}
	\end{figure}
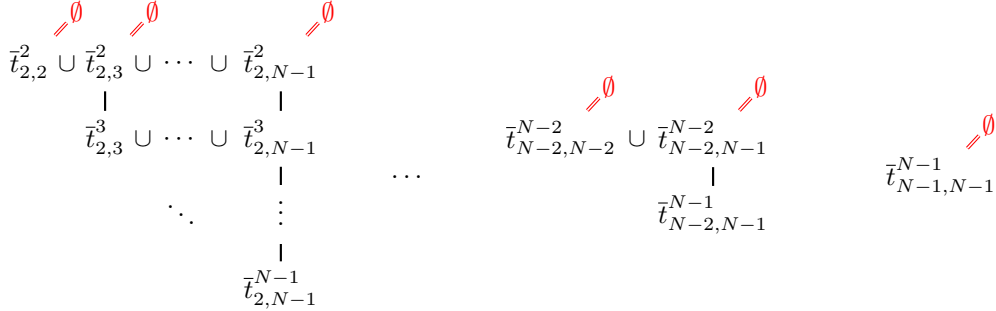
	In Figure~\ref{f:emptyset}, the sets connected by straight lines must have the same size. 
	As a result, every set in that diagram is equal to the empty set, and the only remaining sets of parameters are of the form $t_{1,j}^{\ell}$, arranged as follows:
	\begin{center}
		\begin{tikzpicture}
			\foreach \y in {1,2,3}
			{
				\foreach \x in {\y,...,4}
				{
					\ifthenelse{\x<4}{
						\node [anchor=center] at (\x,6-\y) {$\ol{t}^\y_{1,\x}$};
						\node [anchor=center]at (\x+0.5,6-\y) {$\cup$};
					}{
						\node [anchor=center]at (\x,6-\y) {$\cdots$};
						\node [anchor=center] at (\x+0.5,6-\y) {$\cup$};
					}
				}
				\node[anchor=west] at (4.7, 6-\y) {$\ol{t}^{\y}_{1,N-1}$};
				\node[anchor=west] at (6.7,6-\y) {\large $=\ol{t}^{\y}$};
			}
			\node [anchor=center]at (4,2) {$\ddots$};
			\node [anchor=center]at (5.5,2) {$\vdots$};
			\node [anchor=center]at (7.0,2) {$\vdots$};
			\node [anchor=west]at (4.7,1) {$\ol{t}^{N-1}_{1,N-1}$};	
			\node[anchor=west] at (6.7,1) {\large $=\ol{t}^{N-1}.$};
		\end{tikzpicture}
	\end{center}

	Taking this into account and using the notations for the polynomial version, the Bethe vector simplifies as follows:
	\begin{multline}
		\Psi
		=
		\left(\ol{t}^1-\ol{w}\right)\prod_{\ell=2}^{N-1} \left(\ol{t}^{\ell}-\ol{t}^{\ell-1}\right)
		\sum_{\mathrm{part}}
			\prod_{k=1}^{N-1} \prod_{j < j'}
			\frac{q\ol{t}^k_{1,j'}-q^{-1}\ol{t}^k_{1,j}}{\ol{t}^k_{1,j'}-\ol{t}^k_{1,j}}
			\\
			\times
			\prod_{k=2}^{N-1} \left(\prod_{j < j'} \frac{q\ol{t}^k_{1,j}-q^{-1}\ol{t}^{k-1}_{1,j'}}{\ol{t}^k_{1,j}-\ol{t}^{k-1}_{1,j'}}
			\prod_{k \le j} \frac{K^{(l)}(\ol{t}^k_{1,j}|\ol{t}^{k-1}_{1,j})}{\ol{t}^k_{1,j}-\ol{t}^{k-1}_{1,j}}\right)
				\prod^{\longleftarrow}_{N \geq j > 1} \frac{\tilde{T}_{1j}(\ol{t}^1_{1,j-1}|\ol{w})}{\ol{t}^1_{1,j-1}-\ol{w}}
			 e_{1^n}.
	\end{multline}

	We now distribute the initial factors among the later factors
	\[
		\left(\ol{t}^1-\ol{w}\right)\prod_{\ell=2}^{N-1} \left(\ol{t}^{\ell}-\ol{t}^{\ell-1}\right)
		=
		\prod_{j=1}^{N-1} \left(\ol{t}^1_{1,j}-\ol{w}\right)\prod_{\ell=2}^{N-1}\prod_{j=\ell}^{N-1}\prod_{k=\ell-1}^{N-1} \left(\ol{t}^{\ell}_{1,j}-\ol{t}^{\ell-1}_{1,k}\right),
	\]
	to get
	\begin{multline}
		\Psi
		=
		\sum_{\mathrm{part}}
		\prod_{j=1}^{N-1} \cancel{\left(\ol{t}^1_{1,j}-\ol{w}\right)}\prod_{\ell=2}^{N-1}\prod_{j=\ell}^{N-1}\prod_{\substack{k=\ell-1 \\ k\neq j}}^{N-1} \left(\ol{t}^{\ell}_{1,j}-\ol{t}^{\ell-1}_{1,k}\right)
			\prod_{k=1}^{N-1} \prod_{j < j'}
			\frac{q\ol{t}^k_{1,j'}-q^{-1}\ol{t}^k_{1,j}}{\ol{t}^k_{1,j'}-\ol{t}^k_{1,j}}
			\\
			\times
			\prod_{k=2}^{N-1} \left(\prod_{j < j'} \frac{q\ol{t}^k_{1,j}-q^{-1}\ol{t}^{k-1}_{1,j'}}{\ol{t}^k_{1,j}-\ol{t}^{k-1}_{1,j'}}
			\prod_{k \le j} \frac{K^{(l)}(\ol{t}^k_{1,j}|\ol{t}^{k-1}_{1,j})}{\cancel{\ol{t}^k_{1,j}-\ol{t}^{k-1}_{1,j}}}\right)
				\prod^{\longleftarrow}_{N \geq j > 1} \frac{\tilde{T}_{1j}(\ol{t}^1_{1,j-1}|\ol{w})}{\cancel{\ol{t}^1_{1,j-1}-\ol{w}}}
			 e_{1^n}.
	\end{multline}

	There is another cancellation. Write
	\[
		\prod_{j=\ell}^{N-1}\prod_{\substack{k=\ell-1 \\ k\neq j}}^{N-1} \left(\ol{t}^{\ell}_{1,j}-\ol{t}^{\ell-1}_{1,k}\right)
		=
		\prod_{j=\ell}^{N-1}\prod_{\ell-1\leq k < j} \left(\ol{t}^{\ell}_{1,j}-\ol{t}^{\ell-1}_{1,k}\right)
		\prod_{j=\ell}^{N-1}\prod_{j < k \leq N-1} \left(\ol{t}^{\ell}_{1,j}-\ol{t}^{\ell-1}_{1,k}\right).
	\]
	The second product here cancels with the denominator:
	\begin{multline}
		\Psi
		=
		\sum_{\mathrm{part}}\prod_{k=2}^{N-1}\prod_{k \leq j' < j \leq N-1} \left(\ol{t}^k_{1,j}-\ol{t}^{k-1}_{1,j'}\right)
			\prod_{k=1}^{N-1} \prod_{j < j'}
			\frac{q\ol{t}^k_{1,j'}-q^{-1}\ol{t}^k_{1,j}}{\ol{t}^k_{1,j'}-\ol{t}^k_{1,j}}
			\\
			\times
			\prod_{k=2}^{N-1} \left(\prod_{j < j'} 
			\left(q\ol{t}^k_{1,j}-q^{-1}\ol{t}^{k-1}_{1,j'}\right)
			\prod_{k \le j} K^{(l)}(\ol{t}^k_{1,j}|\ol{t}^{k-1}_{1,j})\right)
				\prod^{\longleftarrow}_{N \geq j > 1} \tilde{T}_{1j}(\ol{t}^1_{1,j-1}|\ol{w})
			 e_{1^n}.
	\end{multline}

	We now specialize the sets of variables to 
	\begin{align*}
		\ol{t}^1 &= \ol{w}_{I_2} \cup \ol{w}_{I_3} \cup \cdots \cup \ol{w}_{I_N}
		\\
		\ol{t}^2 &= \ol{w}_{I_3} \cup \cdots \cup \ol{w}_{I_N}
		\\
		&\vdots 
		\\
		\ol{t}^{N-1} &= \ol{w}_{I_N}.
	\end{align*}
	Observe that there are factors in the numerator which lead to zeros for certain configurations of parameters:
	\[
		\prod_{k=2}^{N-1}\prod_{k \leq j' < j \leq N-1} \left(\ol{t}^k_{1,j}-\ol{t}^{k-1}_{1,j'}\right).
	\] 
	This has the following effect: each set is ``connected to'' all sets that are strictly to the left of it in the row directly above it. 
	If these connected sets share any elements, then the term corresponding to that partition vanishes.
	See Figure~\ref{f:sets-spec-example} for an example.
	\begin{figure}[ht]
		\centering
		\begin{tikzpicture}
			\foreach \y in {1,2,3}
			{
				\foreach \x in {\y,...,3}
				{
					\node [anchor=center] (t\x\y) at (\x,6-\y) {$\ol{t}^\y_{1,\x}$};
					\node [anchor=center]at (\x+0.5,6-\y) {$\cup$};
				}
				\node[anchor=center] (t4\y) at (4,6-\y) {$\ol{t}^{\y}_{1,4}$};
			}
			\node [anchor=center] (t44) at (4,2) {$\ol{t}^{4}_{1,4}$};
			
			\begin{scope}[thick]
				\draw (t21) -- (t22);
				\draw (t31) -- (t32);
				\draw (t41) -- (t42);
				\draw (t32) -- (t33);
				\draw (t42) -- (t43);
				\draw (t43) -- (t44);				
			\end{scope}

			\begin{scope}[dashed, purple]
				\draw (t44) -- (t33);
				\draw (t33) -- (t22);
				\draw (t22) -- (t11);

				\draw (t43) -- (t22);
				\draw (t43) -- (t32);
				\draw (t32) -- (t21);
				\draw (t32) -- (t11);

				\draw (t42) -- (t11);
				\draw (t42) -- (t21);
				\draw (t42) -- (t31);
			\end{scope}
		\end{tikzpicture}
		\caption{Visualization of sets of parameters for $N=5$. Sets connected by solid lines must have the same size. Sets connected by red dashed lines must not share any elements. }
		\label{f:sets-spec-example}
	\end{figure}
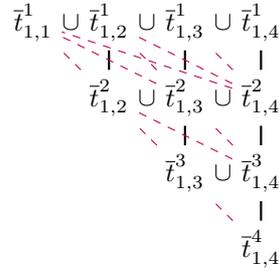

	As a result, the only possible configuration is $\ol{t}^j_{1,k} = \ol{w}_{I_{k+1}}$ for each set.
	Indeed, this can be seen by working inductively from the right-hand side of the diagram. 

	\begin{center}
		\begin{tikzpicture}
			\foreach \y in {1,2,3}
			{
				\foreach \x in {\y,...,3}
				{
					\node [anchor=center] at (\x,6-\y) {$\ol{t}^\y_{1,\x}$};
					\node [anchor=center]at (\x+0.5,6-\y) {$\cup$};
				}
				\node[anchor=center] at (4,6-\y) {$\ol{t}^{\y}_{1,4}$};
			}
			\node [anchor=center]at (4,2) {$\ol{t}^{4}_{1,4}$};

			\tikzset{shift={(5,1)}}
			\node at (0,3) {$=$};

			\foreach \y in {2,3,4}
			{
				\foreach \x in {\y,...,4}
				{
					\node [anchor=center] at (\x-1,6-\y) {$\ol{w}_{I_{\x}}$};
					\node [anchor=center]at (\x-1+0.5,6-\y) {$\cup$};
				}
				\node[anchor=center] at (4,6-\y) {$\ol{w}_{I_{5}}$};
			}
			\node [anchor=center]at (4,1) {$\ol{w}_{I_5}$};
		\end{tikzpicture}
	\end{center}
	
	We now give the details of the effect of this specialization. 
	Substituting in the parameters, we see that the Izergin-Korepin determinant specializes as in \eqref{factorizationIZrl}:
	\begin{multline}
		\Psi
		=
		\prod_{k=2}^{N-1}\prod_{k \leq j' < j \leq N-1} \left(\ol{w}_{I_{j+1}}-\ol{w}_{I_{j'+1}}\right)
			\prod_{k=1}^{N-1} \prod_{k \leq j < j' \leq N-1}
			\frac{q\ol{w}_{I_{j'+1}}-q^{-1}\ol{w}_{I_{j+1}}}{\ol{w}_{I_{j'+1}}-\ol{w}_{I_{j+1}}}
			\\
			\times
			\prod_{k=2}^{N-1} \left(\prod_{k\leq j < j'\leq N-1} 
			\left(q\ol{w}_{I_{j+1}}-q^{-1}\ol{w}_{I_{j'+1}}\right)
			\prod_{j=k}^{N-1} \underbrace{K^{(l)}(\ol{w}_{I_{j+1}}|\ol{w}_{I_{j+1}})}_{(q\ol{w}_{I_{j+1}}-q^{-1}\ol{w}_{I_{j+1}})}\right)
			\\
		\times 
		\tilde{T}_{1N}(\ol{w}_{I_N}|\ol{w})
		\cdots 
		\tilde{T}_{12}(\ol{w}_{I_2}|\ol{w})
		e_{1^n}.
	\end{multline}

	
	We also note the following cancellation:
	\begin{multline}
		\Psi = \cancel{\prod_{k=2}^{N-1} \prod_{k-1\leq j' < j \leq N-1} \left(\ol{w}_{I_{j+1}}-\ol{w}_{I_{j'+1}}\right)}
		\prod_{j=1}^{N-2} \prod_{j'=j+1}^{N-1} \left(\frac{q \ol{w}_{I_{j'+1}}-q^{-1}\ol{w}_{I_{j+1}}}{\cancel{\ol{w}_{I_{j'+1}}-\ol{w}_{I_{j+1}}}}\right)^j
		\\
		\times
		\prod_{j=1}^{N-2} \prod_{j'=j+1}^{N-1} \left(q \ol{w}_{I_{j+1}}-q^{-1}\ol{w}_{I_{j'+1}}\right)^{j-1}
		\prod_{j=2}^{N-1} \left(q \ol{w}_{I_{j+1}}-q^{-1} \ol{w}_{I_{j+1}}\right)^{j-1}
		\\
		\times 
		\tilde{T}_{1N}(\ol{w}_{I_N}|\ol{w})
		\cdots 
		\tilde{T}_{12}(\ol{w}_{I_2}|\ol{w})
		e_{1^n}.
	\end{multline}

As a result,  we have the following final expression.

\begin{proposition}
Specializing the sets of variables to 
	\begin{align*}
		\ol{t}^j &= \ol{w}_{I_{j+1}} \cup \ol{w}_{I_{j+2}} \cup \cdots \cup \ol{w}_{I_N},
		\ \ \ j=1,\dots,N-1,
	\end{align*}
we have
	\begin{multline} \label{psi-expr1-final}
		\Psi = 
		\prod_{1 \le j < k \le N-1} \left(q \ol{w}_{I_{k+1}}-q^{-1}\ol{w}_{I_{j+1}}\right)^j
		\left(q \ol{w}_{I_{j+1}}-q^{-1}\ol{w}_{I_{k+1}}\right)^{j-1}
		\prod_{j=2}^{N-1} \left(q \ol{w}_{I_{j+1}}-q^{-1} \ol{w}_{I_{j+1}}\right)^{j-1}
		\\
		\times 
		\tilde{T}_{1N}(\ol{w}_{I_N}|\ol{w})
		\cdots 
		\tilde{T}_{12}(\ol{w}_{I_2}|\ol{w})
		e_{1^n}.
	\end{multline}

	\end{proposition}

\subsection{Second specialization} \label{sec:secondspec}

We now proceed with the same steps as the previous subsection, but with another expression.
We take the same vector representation of the universal Bethe vector $\widehat{B}(\bar{t}^{\,1},\bar{t}^{\,2},\dots,\bar{t}^{\,N-1})$, multiply by the same overall factor $\prod_{\ell=2}^{N-1} \left(\ol{t}^\ell-\ol{t}^{\ell-1}\right)$
and act on the highest weight vector
\begin{align}
	\widetilde{\Psi} &:= \left(\ol{t}^1-\ol{w}\right)
	\prod_{\ell=2}^{N-1} \left(\ol{t}^\ell-\ol{t}^{\ell-1}\right) 
	\widehat{B}(\bar{t}^{\,1},\bar{t}^{\,2},\dots,\bar{t}^{\,N-1}) e_{1^n} \nonumber \\
	&= \left(\ol{t}^1-\ol{w}\right)
	\prod_{\ell=2}^{N-1} \left(\ol{t}^\ell-\ol{t}^{\ell-1}\right) 
	\sum_{\mathrm{part}}
	\prod_{k=1}^{N-1} \prod_{i,j \prec i',j'}
	\frac{q \ol{t}^k_{i',j'}-q^{-1} \ol{t}^k_{i,j}}{\ol{t}^k_{i',j'}-\ol{t}^k_{i,j}} \nonumber \\
	&\times
	\prod_{k=2}^{N-1} \left(\prod_{i,j \prec^t i',j'} \frac{q\ol{t}^k_{i,j}-q^{-1} \ol{t}^{k-1}_{i',j'}}{\ol{t}^k_{i,j}-\ol{t}^{k-1}_{i',j'}}
	\prod_{i<j} K^{(r),\mathrm{rat}}(\ol{t}^k_{i,j}|\ol{t}^{k-1}_{i,j})\right)
	\nonumber
	\\
	&\times
	\prod^{\longleftarrow}_{N-1 \geq k \geq 1} 
	\left(
		\prod^{\longrightarrow}_{1 \leq j \leq k} \tilde{T}^{\mathrm{rat}}_{j,k+1}(\ol{t}^k_{j,k};\overline{w})
	\right)
	\prod_{k=1}^{N-2}\prod_{k,k\prec^t i,j} \tilde{T}^{\mathrm{rat}}_{k+1,k+1}(\ol{t}^k_{i,j};\overline{w})
	e_{1^n}.
\end{align}
Observe that, for the vector representation,
\[
	\prod_{k=1}^{N-2}\prod_{k,k\prec^t i,j} \tilde{T}^{\mathrm{rat}}_{k+1,k+1}(\ol{t}^k_{i,j};\overline{w})
	 e_{1^n} = e_{1^n}.
\]
Using this action and
rewriting using the polynomial version of the domain wall boundary partition functions
and the monodromy matrices, we have
\begin{align}
	\widetilde{\Psi} &= 
	\left(\ol{t}^1-\ol{w}\right)\prod_{\ell=2}^{N-1} \left(\ol{t}^\ell-\ol{t}^{\ell-1}\right)
	\sum_{\mathrm{part}}
		\prod_{k=1}^{N-1} \prod_{i,j \prec^t i',j'}
		\frac{q\ol{t}^k_{i',j'}-q^{-1}\ol{t}^k_{i,j}}{\ol{t}^k_{i',j'}-\ol{t}^k_{i,j}}
		\nonumber
		\\
		&\times
		\prod_{k=2}^{N-1} \left(\prod_{i,j \prec^t i',j'} \frac{q\ol{t}^k_{i,j}-q^{-1}\ol{t}^{k-1}_{i',j'}}{\ol{t}^k_{i,j}-\ol{t}^{k-1}_{i',j'}}
		\prod_{i<j} \frac{K^{(r)}(\ol{t}^k_{i,j}|\ol{t}^{k-1}_{i,j})}{\ol{t}^k_{i,j}-\ol{t}^{k-1}_{i,j}}\right)
		\prod^{\longleftarrow}_{N-1 \geq k \geq 1} 
		\left(
			\prod^{\longrightarrow}_{1 \leq j \leq k} \frac{\tilde{T}_{j,k+1}(\ol{t}^k_{j,k};\ol{w})	}{\ol{t}^k_{j,k}-\ol{w}}	\right)
		e_{1^n}.
\end{align}
After a small cancellation, we get
\begin{align}
	\widetilde{\Psi} &= 
	\prod_{\ell=2}^{N-1} \left(\ol{t}^\ell-\ol{t}^{\ell-1}\right)
	\sum_{\mathrm{part}}
		\frac{\prod_{j=2}^{N-1}\left(\ol{t}^1_{1j}-\ol{w}\right)}{\prod_{N-1\geq k \geq 1} \prod_{2 \leq j \leq k}
			\left(\ol{t}^k_{j,k}-\ol{w}\right)}
		\prod_{k=1}^{N-1} \prod_{i,j \prec^t i',j'}
		\frac{q\ol{t}^k_{i',j'}-q^{-1}\ol{t}^k_{i,j}}{\ol{t}^k_{i',j'}-\ol{t}^k_{i,j}}
		\nonumber
		\\
&		\times
		\prod_{k=2}^{N-1} \left(\prod_{i,j \prec^t i',j'} \frac{q\ol{t}^k_{i,j}-q^{-1}\ol{t}^{k-1}_{i',j'}}{\ol{t}^k_{i,j}-\ol{t}^{k-1}_{i',j'}}
		\prod_{i<j} \frac{K^{(r)}(\ol{t}^k_{i,j}|\ol{t}^{k-1}_{i,j})}{\ol{t}^k_{i,j}-\ol{t}^{k-1}_{i,j}}\right)
		\prod^{\longleftarrow}_{N-1 \geq k \geq 1} 
		\left(
			\prod^{\longrightarrow}_{1 \leq j \leq k} \tilde{T}_{j,k+1}(\ol{t}^k_{j,k};\ol{w})
		\right)
		e_{1^n}.
\end{align}

The first step is to specialize to $\ol{t^1} = \ol{w}_{I_2} \cup \cdots \cup \ol{w}_{I_N}$.
Observe that the factor $\left(\ol{t}^1_{1j}-\ol{w}\right)$ will be zero unless the partition satisfies
\[
	\ol{t}^1_{1,1} = \ol{w}_{I_2} \cup \cdots \cup \ol{w}_{I_N};
	\qquad 
	\ol{t}^1_{1,2} = \dots = \ol{t}^1_{1,N-1} = \varnothing.
\]
Then, from the rule that vertically aligned partitions have the same size, we obtain 
\[
	\ol{t}^j_{1,k} = \varnothing \qquad \forall j\geq 1,k>1.
\]

We have factor: 
\[
	\frac{\left(\ol{t}^2-\ol{t}^{1}\right)}{\left(\ol{t}^2_{22}-\ol{w}\right)}
	=
	\frac{\prod_{j=2}^{N-1} \left(\ol{t}^2_{2,j}-\ol{t}^{1}_{11}\right)}{\left(\ol{t}^2_{22}-\ol{w}\right)}
	=
	\frac{\prod_{j=2}^{N-1} \left(\ol{t}^2_{2,j}-\ol{w}_{I_2} \cup \cdots \cup \ol{w}_{I_N}\right)}{\left(\ol{t}^2_{22}-\ol{w}\right)}
	=
	\frac{\prod_{j=3}^{N-1} \left(\ol{t}^2_{2,j}-\ol{w}_{I_2} \cup \cdots \cup \ol{w}_{I_N}\right)}{\left(\ol{t}^2_{22}-\ol{w}_{I_1}\right)}.
\]
We now specialize to $\ol{t}^2 = \ol{w}_{I_3} \cup \cdots \cup \ol{w}_{I_N}$. 
There will be a zero factor in the above product unless 
\[
	t^2_{22} = \ol{w}_{I_2} \cup \cdots \cup \ol{w}_{I_N};
	\qquad 
	t^2_{2,3} = \dots = t^2_{2,N-1} = \varnothing.
\]
As before, this spreads vertically up the partition. 
Inductively, we apply the same argument. 
In each case we specialize to 
\[
	\ol{t}^j = \ol{w}_{I_{j+1}} \cup \cdots \cup \ol{w}_{I_N},
\]
and we find that all summands vanish except the one corresponding to
\[
	\ol{t}^j_{jj} = \ol{w}_{I_{j+1}} \cup \cdots \cup \ol{w}_{I_N};
	\qquad
	\ol{t}^j_{k \ell} = \varnothing, \  j\leq k < \ell \leq N.
\]

With this specialization, we now analyse its effect on the expression for $\widetilde{\Psi}$. 
We have a product
\[
	\frac{\left(\ol{t}^1-\ol{w}\right)\prod_{\ell=2}^{N-1}\left(\ol{t}^\ell-\ol{t}^{\ell-1}\right)}{\prod_{1 \leq j \leq N-1} \left(\ol{t}^j_{jj}-\ol{w}\right)}
	=
	\cancel{\frac{\ol{t}^1_{11}-\ol{w}}{\ol{t}^1_{11}-\ol{w}}}
	\prod_{j=2}^{N-1}
	\frac{\left(\ol{t}^j_{jj}-\ol{t}^{j-1}_{j-1,j-1}\right)}{\left(\ol{t}^j_{jj}-\ol{w}\right)}.
\]
Then, as a product over subsets of $\ol{w}$, this is equal to 
\[
	\prod_{j=2}^{N-1}
	\frac{1}{\left(\ol{w}_{I_{j+1}}\cup \cdots \cup \ol{w}_{I_{N}} -\ol{w}_{I_{1}}\cup \cdots \cup \ol{w}_{I_{j-1}}\right)}.
\]
All other factors disappear, and
we then are left with the following expression.

\begin{proposition}
Specializing the sets of variables to 
	\begin{align*}
		\ol{t}^j &= \ol{w}_{I_{j+1}} \cup \ol{w}_{I_{j+2}} \cup \cdots \cup \ol{w}_{I_N},
		\ \ \ j=1,\dots,N-1,
	\end{align*}
we have
\begin{multline} \label{psi-expr2-final}
	\widetilde{\Psi}
	=
	\prod_{j=2}^{N-1}
	\frac{1}{\left(\ol{w}_{I_{j+1}}\cup \cdots \cup \ol{w}_{I_{N}} -\ol{w}_{I_{1}}\cup \cdots \cup \ol{w}_{I_{j-1}}\right)}
	\\
	\times
	\tilde{T}_{N-1,N}(\ol{w}_{I_N};\ol{w})
	\tilde{T}_{N-2,N-1}(\ol{w}_{I_{N-1}}\cup\ol{w}_{I_N};\ol{w})
	\cdots
	\tilde{T}_{1,2}(\ol{w}_{I_{2}}\cup \cdots \cup\ol{w}_{I_N};\ol{w}) e_{1^n}.
\end{multline}

\end{proposition}



\section{Yangian case}
We present analogous results obtained in previous sections to the case of the Yangian.
The Yangian $Y(\mathfrak{gl}_N)$ is a unital associative algebra generated by the coefficients of the formal series
\[
T(x)=\sum_{i,j=1}^N E_{ij}\otimes T_{ij}(x),\qquad
T_{ij}(x)=\sum_{r\ge0}T_{ij}[r] x^{-r},
\]
subject to the $RTT$ relation
\begin{equation}\label{eq:RTT-Yangian}
R(x,y)\, T_1(x) T_2(y) = T_2(y) T_1(x)\, R(x,y),
\end{equation}
where $T_1(x)=T(x)\otimes 1$, $T_2(y)=1\otimes T(y)$, and
\begin{align}
R(x,y)=(x-y+h)\sum_{i=1}^N E_{ii}\otimes E_{ii}
+(x-y)\sum_{i\neq j}E_{ii}\otimes E_{jj}
+h \sum_{i\neq j}E_{ij}\otimes E_{ji},
\label{rationalRmatrix}
\end{align}
is the rational $R$-matrix (Figure \ref{rationalrmatrixfigure}).  

\begin{figure}
\begin{center}
	\includegraphics{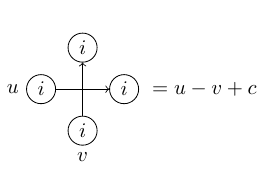}
	\includegraphics{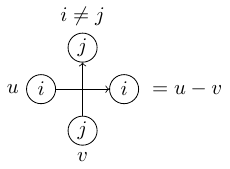}
	\includegraphics{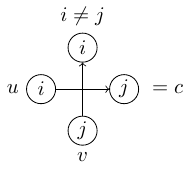}
\end{center} 
\caption{Non-zero matrix elements of  the rational $R$-matrix.} \label{rationalrmatrixfigure}
\end{figure}

The rational $R$-matrix satisfies the Yang-Baxter relation
\begin{align}
R_{12}(x,y)\,R_{13}(x,z)\,R_{23}(y,z)
&=R_{23}(y,z)\,R_{13}(x,z)\,R_{12}(x,y), \label{eq:YB-rational} 
\end{align}
and the unitarity relation becomes
\begin{align}
R_{12}(x,y)\,R_{21}(y,x)&=(x-y+h)(y-x+h)\,\mathrm{I} \otimes \mathrm{I}. \label{eq:unitarity-rational}
\end{align}

The natural representation of $T(x)$ is given by the product of rational $R$-matrices
\[
T^{\mathrm{vect}}(x;\overline{\xi})=
R_{0n}(x,\xi_n)\cdots R_{01}(x,\xi_1),
\]
where $R(x,y)$ is \eqref{rationalRmatrix}, $V_0$ is the auxiliary space, and
$V_1,\dots,V_n$ are quantum spaces with spectral parameters
$\xi_1,\dots,\xi_n$.
By repeated use of the Yang--Baxter equation, one checks that
$T^{\mathrm{vect}}(x;\overline{\xi})$ satisfies the $RTT$ relation
\eqref{eq:RTT-Yangian}.

We study the multiple version of the following fundamental
commutation relations
\begin{align}
T_{ik}(x)T_{ij}(y)&=\frac{y-x+h}{y-x}T_{ij}(y)T_{ik}(x)-\frac{h}{y-x}T_{ij}(x)T_{ik}(y), \ \ \  j \neq k, \label{rationalfundcommone} \\
[T_{ij}(x),T_{ij}(y)]&=0, \label{rationalfundcommtwo}
\end{align}
for $i,j,k=1,\dots,N$.

\color{red}



\color{black}

We present analogous results to the Yangian case,
replacing the $u$-, $v$- and $w$-variables for the special functions
and partition functions in the previous sections
to $x$-, $y$- and $z$-variables.
This replacement also implies that we replace the $R$-matrix 
which constructs the corresponding partition functions
from $R(u.v)$ or $\tilde{R}(u,v)$ to $R(x,y)$ \eqref{rationalRmatrix}.
We also do not use the symbol $ \tilde{} $ for the rational version as there is no disctinction
for the Yangian case.

We introduce the rational weight functions.
\begin{definition} The rational weight functions are defined as
\begin{align} \label{rational-wavefunction}
        &W (\overline{x}^1,\dots,\overline{x}^{N-1} | \overline{y} | \bm I ) 
        = \sum_{\sigma_1 \in S_{k_1}} \cdots \sum_{\sigma_{N-1} \in S_{k_{N-1}}} \nonumber \\
        &\prod_{p=1}^{N-2} 
\Bigg\{
            \prod_{a=1}^{k_p} \Bigg(
\prod_{i=1}^{\widetilde{I}_{a}^{(p)}-1} 
 \left(x^{(p)}_{\sigma_p(a)}- x^{(p+1)}_{\sigma_{p+1}(i)}   \right) \times h
                \times 
                \prod_{i=\widetilde{I}_{a}^{(p)}+1}^{ k_{p+1}}
                    \left(x^{(p)}_{\sigma_p(a)}-x^{(p+1)}_{\sigma_{p+1}(i)}
+h \right)
            \Bigg) \nonumber \\
&\times
            \prod_{a<b}^{k_p} 
                \frac{x^{(p)}_{\sigma_p(a)}-x^{(p)}_{\sigma_p(b)}-h}{x^{(p)}_{\sigma_p(a)}-x^{(p)}_{\sigma_p(b)}}
        \Bigg\}
\nonumber        \\ 
&\times 
        \prod_{a=1}^{k_{N-1}} \left(
            \prod_{i=1}^{I_a^{(N-1)}-1} 
                \left(x_{\sigma_{N-1}(a)}^{(N-1)}-y_i^{(N-1)}\right) \times h \times
            \prod_{i=I_a^{(N-1)}+1}^{L_{N-1}} 
                \left(x_{\sigma_{N-1}(a)}^{(N-1)}-y_i^{(N-1)}+h \right)
        \right)
     \nonumber   \\ 
&\times 
        \prod_{a<b}^{k_{N-1}} 
            \frac{x^{(N-1)}_{\sigma_{N-1}(a)}-x^{(N-1)}_{\sigma_{N-1}(b)}-h}{x^{(N-1)}_{\sigma_{N-1}(a)}-x^{(N-1)}_{\sigma_{N-1}(b)}}.
    \end{align}   
    \end{definition}

The correspondence with the partition functions is given by the following relation.

\begin{theorem} \label{offshellbetherationalweight}
The following holds:
\begin{align}
\psi(\overline{x}^1,\dots,\overline{x}^{N-1} | \overline{y} | \bm I ) =
W (\overline{x}^1,\dots,\overline{x}^{N-1} | \overline{y} | \bm I ).
\end{align}
\end{theorem}

The domain wall boundary partition functions
for the Yangian version
\begin{align}
&H(\overline{x}|\overline{y}) 
:=
e_{1^{|\overline{x}|}}^*
T_{21}(\overline{x};\overline{y}) 
e_{2^{|\overline{x}|}},
\end{align}
where $|\overline{x}|=|\overline{y}|$,
can be expressed by the following version of the
Izergin-Korepin determinant
\begin{align}
K(\overline{x}|\overline{y})
&:= 
		\frac{\prod_{1\leq i,j \leq n}(x_i-y_j+h)(x_i-y_j)}{\prod_{1\leq i <j \leq n}(x_i-x_j)(y_j-y_i)}
		\det_{1 \le i,j \le n} \left[\frac{h}{(x_i-y_j+h)(x_i-y_j)}\right] \nonumber \\
		&=
		\frac{1}{\prod_{1\leq i <j \leq n}(x_i-x_j)(y_j-y_i)}
		\det_{1 \le i,j \le n} \left[
		h
		\prod_{\substack{k=1 \\ k \neq j}}^n
		(x_i-y_k+h)(x_i-y_k)
		\right],
\end{align}
for $\overline{x}=\{x_1,x_2,\dots,x_n  \}$, $\overline{y}=\{y_1,y_2,\dots,y_n \}$.
\begin{theorem}
The following holds:
\begin{align}
H(\overline{x}|\overline{y}) = K(\overline{x}|\overline{y}).
\label{domainwallpartitionIKdetrational}
\end{align}
\end{theorem}

We now present analogous results obtained in previous sections.
The following is the Yangian analogue of
Theorem
\ref{thmmultiplecommutationrelations}.

\begin{theorem} \label{thmmultiplecommutationrelationsrational}
The following commutation relation holds:
\begin{align}
&T_{N1}(\overline{x}^1) T_{N2}(\overline{x}^2) \cdots T_{NN}(\overline{x}^N)
=\sum_{\{ \overline{x}^1, \overline{x}^2,\dots,\overline{x}^N \} \mapsto \{ \overline{y}^1,
\overline{y}^2,\dots,\overline{y}^N \} } 
\frac{1}{\prod_{1 \le j < k \le N} (\overline{y}^k-\overline{y}^{j})( \overline{y}^j- \overline{y}^{k-1}+h)}
\nonumber \\
\times&
\frac{\prod_{j=1}^{N-1} ( \overline{x}^N- \overline{y}^j+h)}{
\prod_{j=1}^{N-1} ( \overline{x}^j- \overline{y}^N+h)
\prod_{\ell=1}^{N-2}
\prod_{j=1}^\ell \prod_{k=1}^{\ell+1} ( \overline{x}^j- \overline{y}^k+h)}
\nonumber \\
\times&W(
\overline{x}^1,\overline{x}^1 \cup \overline{x}^2,\dots,\overline{x}^1 \cup \cdots \cup \overline{x}^{N-1}
|\overline{y}^1,\overline{y}^2,\dots,\overline{y}^{N-1},\overline{y}^N|1^{|\overline{x}^1|},2^{|\overline{x}^2|},\dots,(N-1)^{|\overline{x}^{N-1}|},N^{|\overline{x}^{N}|} )
\nonumber \\
\times&
T_{NN}(\overline{y}^N) \cdots T_{N2}(\overline{y}^2)  T_{N1}(\overline{y}^1).
\end{align}
\end{theorem}

Theorem \ref{thmmultiplecommutationrelationsrational}
follows from combining
Proposition \ref{keypropositionforproofrational}
and Proposition \ref{keyrelationforproofrational}
given below.

\begin{proposition} \label{keypropositionforproofrational}
The following commutation relation holds:
\begin{align}
&T_{N1}(\overline{x}^1) T_{N2}(\overline{x}^2) \cdots T_{NN}(\overline{x}^N)
=\sum_{\{ \overline{x}^1, \overline{x}^2,\dots,\overline{x}^N \} \mapsto \{ \overline{y}^1,
\overline{y}^2,\dots,\overline{y}^N \} } 
\frac{1}{\prod_{1 \le j < k \le N} (\overline{y}^k-\overline{y}^{j})( \overline{y}^j- \overline{y}^{k-1}+h)}
\nonumber \\
\times&
\prod_{j=1}^{N-1} (\overline{x}^N- \overline{y}^j+h)
H(\overline{x}^1, \overline{x}^2,\dots,\overline{x}^{N-1}|\overline{y}^1,\overline{y}^2,\dots,\overline{y}^{N-1})
T_{NN}(\overline{y}^N) \cdots T_{N2}(\overline{y}^2)  T_{N1}(\overline{y}^1), 
\end{align}
where
\begin{align}
&H(\overline{x}^1, \overline{x}^2,\dots,\overline{x}^{N-1}|\overline{y}^1,\overline{y}^2,\dots,\overline{y}^{N-1}) \nonumber \\
=&
e_{1^{|\overline{x}^1|},2^{|\overline{x}^2|},\dots,(N-1)^{|\overline{x}^{N-1}|}}^*
T_{N1}(\overline{x}^1;\overline{y}^1,\overline{y}^2,\dots,\overline{y}^{N-1}) 
T_{N2}(\overline{x}^2;\overline{y}^1,\overline{y}^2,\dots,\overline{y}^{N-1}) 
\nonumber \\
&\times \cdots \times
T_{NN-1}(\overline{x}^{N-1};\overline{y}^1,\overline{y}^2,\dots,\overline{y}^{N-1}) 
e_{N^{|\overline{x}^1|+|\overline{x}^2|+\cdots+|\overline{x}^{N-1}|}}.
\label{commcoeffratpartrational}
\end{align}
\end{proposition}

\begin{proposition} \label{keyrelationforproofrational}
The following holds:
\begin{align}
&H(\overline{x}^1, \overline{x}^2,\dots,\overline{x}^{N-1}|\overline{y}^1,\overline{y}^2,\dots,\overline{y}^{N-1})
\nonumber \\
=&
\frac{1}{
\prod_{j=1}^{N-1} ( \overline{x}^j- \overline{y}^N+h)
\prod_{\ell=1}^{N-2}
\prod_{j=1}^\ell \prod_{k=1}^{\ell+1} ( \overline{x}^j- \overline{x}^k+h)}
\nonumber \\
\times& W(
\overline{x}^1,\overline{x}^1 \cup \overline{x}^2,\dots,\overline{x}^1 \cup \cdots \cup \overline{x}^{N-1}
|\overline{y}^1,\overline{y}^2,\dots,\overline{y}^{N-1},\overline{y}^N|1^{|\overline{u}^1|},2^{|\overline{u}^2|},\dots,(N-1)^{|\overline{u}^{N-1}|},N^{|\overline{u}^{N}|} ).
\end{align}

\end{proposition}

Proposition
\ref{keyrelationforproofrational}
follows from combining
Lemma \ref{slightlylargerrational} 
and Proposition \ref{relateoffshellbetherectangularrational}.
For description, we 
introduce the following partition functions
\begin{align}
&K
(\overline{x}^1,\overline{x}^2,\dots,\overline{x}^{N-1}
|\overline{y}^1,\overline{y}^2,\dots,\overline{y}^{N-1},\overline{y}^N)
\nonumber \\
:=&e_{1^{|\overline{x}^1|},2^{|\overline{x}^2|},\dots,(N-1)^{|\overline{x}^{N-1}|},N^{|\overline{x}^{N}|} }^*
T_{N1}(\overline{x}^1;\overline{y}^1,\overline{y}^2,\dots,\overline{y}^{N-1},\overline{y}^N) 
T_{N2}(\overline{x}^2;\overline{y}^1,\overline{y}^2,\dots,\overline{y}^{N-1},\overline{y}^N) 
\nonumber \\
&\times \cdots \times
T_{NN-1}(\overline{x}^{N-1};\overline{y}^1,\overline{y}^2,\dots,\overline{y}^{N-1},\overline{y}^N) 
e_{N^{|\overline{x}^1|+|\overline{x}^2|+\cdots+|\overline{x}^{N-1}|+|\overline{x}^N|}}.
\end{align}

\begin{lemma} \label{slightlylargerrational}
The following holds:
\begin{align}
&
K
(\overline{x}^1,\overline{x}^2,\dots,\overline{x}^{N-1}
|\overline{y}^1,\overline{y}^2,\dots,\overline{y}^{N-1},\overline{y}^N) \nonumber \\
=&\prod_{j=1}^{N-1} ( \overline{x}^j- \overline{y}^N+h)
H(\overline{x}^1, \overline{x}^2,\dots,\overline{x}^{N-1}|\overline{y}^1,\dots,\overline{y}^{N-1}).
\label{relationtwopartitionfunctionsgridrational}
\end{align}
\end{lemma}

\begin{proposition} \label{relateoffshellbetherectangularrational}
The following holds:
\begin{align}
&W(
\overline{x}^1,\overline{x}^1 \cup \overline{x}^2,\dots,\overline{x}^1 \cup \cdots \cup \overline{x}^{N-1}
|\overline{y}^1,\overline{y}^2,\dots,\overline{y}^{N-1},\overline{y}^N|1^{|\overline{x}^1|},2^{|\overline{x}^2|},\dots,(N-1)^{|\overline{x}^{N-1}|},N^{|\overline{x}^{N}|} ) \nonumber \\
=&
\prod_{\ell=1}^{N-2}
\prod_{j=1}^\ell \prod_{k=1}^{\ell+1} (\overline{x}^j- \overline{y}^k+h)
K(\overline{x}^1, \overline{x}^2,\dots,\overline{x}^{N-1}|\overline{y}^1,\overline{y}^2,\dots,\overline{y}^{N-1},\overline{y}^N).
\end{align}
\end{proposition}
Proposition \eqref{relateoffshellbetherectangularrational} follows as a specific case 
$\boldsymbol{I}
=(1^{|\overline{u}^1|},2^{|\overline{u}^2|},\dots,(N-1)^{|\overline{u}^{N-1}|},N^{|\overline{u}^{N}|})$
of a more generic relation between partition functions, combined with Theorem \ref{offshellbetherationalweight}.

\begin{proposition} 
We have
\begin{align}
&\psi(\overline{z}_{J_1},\overline{z}_{J_1} \cup \overline{z}_{J_2},\dots,\overline{z}_{J_1} \cup \cdots \cup \overline{z}_{J_{N-1}}
|\overline{z}|\boldsymbol{I})
\nonumber \\
=&
\prod_{\ell=1}^{N-2}
\prod_{j=1}^\ell \prod_{k=1}^{\ell+1} ( \overline{z}_{J_j}-\overline{z}_{J_k}+h)
K
(\overline{z}_{J_1},\overline{z}_{J_2},\dots,\overline{z}_{J_{N-1}}
|\overline{z}|\boldsymbol{I}), \label{relatingpartitionfunctionsrational}
\end{align}
where
\begin{align}
K
(\overline{z}_{J_1},\overline{z}_{J_2},\dots,\overline{z}_{J_{N-1}}
|\overline{z}|\boldsymbol{I}) 
=&e_{\boldsymbol{I} }^*
T_{N1}(\overline{z}_{J_1};\overline{z}) 
T_{N2}(\overline{z}_{J_2};\overline{z}) 
\times \cdots \times
T_{NN-1}(\overline{z}_{J_{N-1}};\overline{z}) 
e_{N^{n}},
\end{align}
with $e_{\boldsymbol{I} }^*=e^*_{i_1,i_2,\dots,i_n}$ for $I=(i_1,i_2,\dots,i_n)$.

\end{proposition}

The $N=2$  case can also be written using the rational version of the Izergin-Korepin determinant as
\begin{align}
&T_{21}(\overline{x}^1) T_{22}(\overline{x}^2) 
=\sum_{\{ \overline{x}^1, \overline{x}^2 \} \mapsto \{ \overline{y}^1,
\overline{y}^2 \} } 
\frac{
 \overline{x}^2- \overline{y}^1+h
}{ (\overline{y}^2-\overline{y}^{1})( \overline{y}^1- \overline{y}^{1}+h)}
K(\overline{x}^1|\overline{y}^1)
T_{22}(\overline{y}^2)  T_{21}(\overline{y}^1).
\end{align}

Molev's construction \cite{Molev} shows that
$T_{21}(z_{J_1};\overline{z}) T_{32}(z_{J_1} \cup z_{J_2};\overline{z}) \cdots T_{N,N-1}(z_{J_1} \cup \cdots \cup z_{J_{N-1}};\overline{z}) e_{N^n}
$ for all $J$ give rise to a construction of the Gelfand-Tsetlin basis for the tensor product of vector representation.
The following are the relations
between
$T_{N1}(\overline{z}_{J_1};\overline{z}) 
T_{N2}(\overline{z}_{J_2};\overline{z}) 
\cdots 
T_{NN-1}(\overline{z}_{J_{N-1}};\overline{z}) 
e_{N^{n}}
$ and $T_{21}(z_{J_1};\overline{w}) T_{32}(z_{J_1} \cup z_{J_2};\overline{w}) \cdots T_{N,N-1}(z_{J_1} \cup \cdots \cup z_{J_{N-1}};\overline{z}) e_{N^n}
$.
\begin{proposition}
The following relation holds:
\begin{align}
	&T_{21}(\ol{z}_{J_1};\ol{z})
	T_{32}(\ol{z}_{J_{1}}\cup\ol{z}_{J_2};\ol{z})
	\cdots
	T_{N,N-1}(\ol{z}_{J_{1}}\cup \cdots \cup\ol{z}_{J_{N-1}};\ol{z})e_{N^n}
	\nonumber \\
	=&
	\prod_{j=2}^{N-1}
	\left(\ol{z}_{J_{1}}\cup \cdots \cup \ol{z}_{J_{j-1}} -\ol{z}_{J_{j+1}}\cup \cdots \cup \ol{z}_{J_{N}}\right) 
	\nonumber \\
	&\times
	\prod_{1 \le j < k \le N-1 } \left( \ol{z}_{J_{j}}-\ol{z}_{J_{k}}+h \right)^{N-k}
		\left(\ol{z}_{J_{k}}-\ol{z}_{J_{j}}+h \right)^{N-k-1}
		\nonumber \\
		&\times
	\prod_{j=1}^{N-2} \left(\ol{z}_{J_{j}}- \ol{z}_{J_{j}}+h \right)^{N-j-1}
	T_{N1}(\ol{z}_{J_1};\ol{z}) T_{N2}(\ol{z}_{J_2};\ol{z})
	\cdots 
	T_{N,N-1}(\ol{z}_{J_{N-1}};\ol{z})
	e_{N^n}. \label{relationGZrational}
\end{align}

\end{proposition}

Changing indices, \eqref{relationGZrational}
is equivalent to the following.

\begin{proposition} \label{showrelationproprational}
The following relation holds:
\begin{align}
	&T_{N-1,N}(\ol{z}_{I_N};\ol{z})
	T_{N-2,N-1}(\ol{z}_{I_{N-1}}\cup\ol{z}_{I_N};\ol{z})
	\cdots
	T_{1,2}(\ol{z}_{I_{2}}\cup \cdots \cup\ol{z}_{I_N};\ol{z})e_{1^n}
	\nonumber
	\\
	=&
	\prod_{j=2}^{N-1}
	\left(\ol{z}_{I_{j+1}}\cup \cdots \cup \ol{z}_{I_{N}} -\ol{z}_{I_{1}}\cup \cdots \cup \ol{z}_{I_{j-1}}\right)
	\nonumber
	\\
&	\times 
	\prod_{1 \le j<k \le N-1} \left( \ol{z}_{I_{k+1}}-\ol{z}_{I_{j+1}}+h \right)^j
	\left( \ol{z}_{I_{j+1}}-\ol{z}_{I_{k+1}}+h \right)^{j-1}
	\nonumber
	\\
&	\times 
	\prod_{j=2}^{N-1} \left( \ol{z}_{I_{j+1}}- \ol{z}_{I_{j+1}}+h \right)^{j-1}
	T_{1N}(\ol{z}_{I_N};\ol{z})
	\cdots 
	T_{12}(\ol{z}_{I_2};\ol{z})
	e_{1^n}. \label{showrelationGZrational}
\end{align}

\end{proposition}


\begin{proof}

We derive \eqref{showrelationGZrational}
as a degeneration from the relation for the quantum affine algebra case \eqref{showrelationGZ}.
First, recall the degeneration process from the trigonometric $R$-matrix
$\tilde{R}(u,v)$ to the rational one $R(x,y)$.
Introducing $x$, $y$ by $u=e^{ \epsilon x}$, $v=e^{\epsilon y}$, $q=e^{ \epsilon h/2 }$
and taking $\epsilon \to 0$, we get
$\tilde{R}(u,v) =  \epsilon R(x,y) + O(\epsilon)$ where $O(\epsilon)$ denotes higher order terms than $\epsilon$.

Keeping this in mind, we take $w_j =e^{\epsilon z_j}$ and $q=e^{ \epsilon h/2 }$.
Note the total number of $R$-matrices consisting the left-hand side
is
$\alpha:=|\ol{w}_{I_{1}}\cup \cdots \cup \ol{w}_{I_{N}}| \sum_{j=2}^N (j-1)|\ol{w}_{I_j}|$
since each $T^\mathrm{poly}_{N-1,N}(w_k|\ol{w})$ consists of $|\ol{w}|= |\ol{w}_{I_{1}}\cup \cdots \cup \ol{w}_{I_{N}}|$
$R$-matrices. Then the left-hand side of \eqref{showrelationGZ} becomes
\begin{multline}
	\tilde{T}_{N-1,N}(\ol{w}_{I_N};\ol{w})
	\tilde{T}_{N-2,N-1}(\ol{w}_{I_{N-1}}\cup\ol{w}_{I_N};\ol{w})
	\cdots
	\tilde{T}_{1,2}(\ol{w}_{I_{2}}\cup \cdots \cup\ol{w}_{I_N};\ol{w}) e_{1^n}
	\\
	\to \epsilon^{\alpha}
	T_{N-1,N}(\ol{z}_{I_N};\ol{z})
	T_{N-2,N-1}(\ol{z}_{I_{N-1}}\cup\ol{z}_{I_N};\ol{z})
	\cdots
	T_{1,2}(\ol{z}_{I_{2}}\cup \cdots \cup\ol{z}_{I_N};\ol{z}) e_{1^n}
+
	O(\epsilon^{\alpha}). \label{limitfromtrigone}
\end{multline}
Here $O(\epsilon^{\alpha})$ denote higher order terms
than $\epsilon^{\alpha}$.

We can also show the right-hand side becomes
\begin{align}
&\epsilon^{\beta+\gamma}
		\prod_{j=2}^{N-1}
	\left(\ol{z}_{I_{j+1}}\cup \cdots \cup \ol{z}_{I_{N}} -\ol{z}_{I_{1}}\cup \cdots \cup \ol{z}_{I_{j-1}}\right)
	\nonumber
	\\
&	\times 
	\prod_{1 \le j<k \le N-1} \left( \ol{z}_{I_{k+1}}-\ol{z}_{I_{j+1}}+h \right)^j
	\left( \ol{z}_{I_{j+1}}-\ol{z}_{I_{k+1}}+h \right)^{j-1}
	\nonumber
	\\
&	\times 
	\prod_{j=2}^{N-1} \left( \ol{z}_{I_{j+1}}- \ol{z}_{I_{j+1}}+h \right)^{j-1}
	T_{1N}(\ol{z}_{I_N};\ol{z})
	\cdots 
	T_{12}(\ol{z}_{I_2};\ol{z})
	e_{1^n}
	+O(\epsilon^{\beta+\gamma})
	, \label{limitfromtrigtwo}
\end{align}
where
$\beta:=|\ol{w}_{I_{1}}\cup \cdots \cup \ol{w}_{I_{N}}| \sum_{j=2}^N |\ol{w}_{I_j}|$
is the number of $R$-matrices
constructing $T_{1N}(\ol{z}_{I_N};\ol{z})
	\cdots 
	T_{12}(\ol{z}_{I_2};\ol{z})$,
	and
$\gamma:=\sum_{j=2}^{N-1}
	|\ol{w}_{I_{j+1}}\cup \cdots \cup \ol{w}_{I_{N}}||\ol{w}_{I_{1}}\cup \cdots \cup \ol{w}_{I_{j-1}}|
	+
	\sum_{1 \le j < k \le N-1 }	
	(2j-1) |\ol{w}_{I_{k+1}}||\ol{w}_{I_{j+1}}|
	+
	\sum_{j=2}^{N-1} (j-1) |\ol{w}_{I_{j+1}}|^2
	$
is the total degree of the overall factor.
$O(\epsilon^{\beta+\gamma})$ denote higher order terms
than $\epsilon^{\beta+\gamma}$.

One can check $\alpha=\beta+\gamma$ by rewriting
$\gamma$ as

\begin{align}
\gamma
	=&\sum_{1 \le j < k \le N} (k-j-1) |\ol{w}_{I_j} ||\ol{w}_{I_k}|
	+\sum_{2 \le j < k \le N } (2j-3) |\ol{w}_{I_{j}}||\ol{w}_{I_{k}}|
	+\sum_{j=3}^{N} (j-2) |\ol{w}_{I_{j}}|^2 \nonumber \\
	=&\sum_{j=3}^{N} (j-2) |\ol{w}_{I_{j}}|^2+\sum_{k=2}^N (k-2)|\ol{w}_{I_1}||\ol{w}_{I_k}|
	+\sum_{2 \le j < k \le N } (k+j-4) |\ol{w}_{I_{j}}||\ol{w}_{I_{k}}|,
\end{align}

and $\alpha-\beta$
as

\begin{align}
&\alpha-\beta=
|\ol{w}_{I_{1}}\cup \cdots \cup \ol{w}_{I_{N}}| \sum_{j=3}^N (j-2)|\ol{w}_{I_j}| \nonumber \\
=&\sum_{j=3}^N \sum_{k=j+1}^N (j-2)|\ol{w}_{I_j}||\ol{w}_{I_k}|
+\sum_{j=3}^N (j-2) |\ol{w}_{I_j}|^2
+\sum_{j=3}^N \sum_{k=1}^{j-1} (j-2)|\ol{w}_{I_j}||\ol{w}_{I_k}| \nonumber \\
=&\sum_{j=3}^N (j-2) |\ol{w}_{I_j}|^2+\sum_{3 \le j < k \le N} (j-2)|\ol{w}_{I_j}||\ol{w}_{I_k}|
+\sum_{1 \le j < k \le N} (k-2)|\ol{w}_{I_j}||\ol{w}_{I_k}| \nonumber \\
	=&\sum_{j=3}^{N} (j-2) |\ol{w}_{I_{j}}|^2+\sum_{k=2}^N (k-2)|\ol{w}_{I_1}||\ol{w}_{I_k}|
	+\sum_{2 \le j < k \le N } (k+j-4) |\ol{w}_{I_{j}}||\ol{w}_{I_{k}}|.
\end{align}

Comparing \eqref{limitfromtrigone} and \eqref{limitfromtrigtwo}
using $\alpha=\beta+\gamma$, dividing by $\epsilon^{\alpha+\beta}$
and taking $\epsilon \to 0$, we get \eqref{showrelationGZrational}.

\end{proof}

In the case $N=3$,
\begin{align}
	&T_{23}(\overline{z}_{I_3};\overline{z})
	T_{12}(\overline{z}_{I_2} \cup \overline{z}_{I_3};\overline{z})
	  e_{1^n} \nn \\
	=&(\overline{z}_{I_3}-\overline{z}_{I_1})
	(\overline{z}_{I_3}-\overline{z}_{I_2}+h) (\overline{z}_{I_3}-\overline{z}_{I_3}+h)
	T_{13}(\overline{z}_{I_3};\overline{z})
	T_{12}(\overline{z}_{I_2};\overline{z}) e_{1^n}.
	\end{align} \\
	Example:
	$n=3,$ $I_1=\{ 2 \}, I_2=\{ 3 \}, I_3=\{ 1 \}$ \\
	One can check
	\begin{align*}
	T_{23}(z_1;\overline{z})
	T_{12}(z_1;\overline{z})
	T_{12}(z_3;\overline{z}) e_{1^3}
	=h(z_1-z_2)(z_1-z_3+h)
	T_{13}(z_1;\overline{z})
	T_{12}(z_3;\overline{z}) e_{1^3},
	\end{align*}
	by directly computing
	\begin{align*}
	&T_{23}(z_1;\overline{z})
	T_{12}(z_1;\overline{z})
	T_{12}(z_3;\overline{z}) e_{1^3} \nn \\
	=&
	-h^3 (z_1-z_2) (z_1-z_2+h) (z_1-z_3+h)^2 (z_2-z_3) (z_3-z_1+h)
	e_{312} \nn \\
	&+
	h^4 (z_1-z_2) (z_1-z_2+h) (z_1-z_3+h)^2 (z_3-z_1+h)
	e_{321}, 
	\end{align*}
	and
	\begin{align*}
	&T_{13}(z_1;\overline{z})
	T_{12}(z_3;\overline{z}) e_{1^3} \nn \\
	=&
	-h^2 (z_1-z_2+h) (z_1-z_3+h) (z_2-z_3) (z_3-z_1+h)
	e_{312}   \nn \\
	&+h^3 (z_1-z_2+h) (z_1-z_3+h) (z_3-z_1+h)
	e_{321}.
	\end{align*}

\section{Conclusion}
We presented in this paper an approach to study
multiple commutation relations
of the quantum affine algebra $U_q(\widehat{\mathfrak{gl}}_N)$.
Our approach uses a graphical description
and is conceptual in the sense that it explains why
the trigonometric weight functions appear as coefficients of all summands.
For the rank one case, this also explains why the coefficients can also be expressed
using the Izergin-Korepin determinants.
It would be interesting to 
investigate other types of commutation relations or
different algebras such as the
Faddeev-Zamolodchikov algebras.
Studying by other means, such as the
$q$-vertex operator, is also interesting.

Another interesting topic is to explore applications.
As for the degenerate case, multiple commutation relations
were used in  \cite{IMO} to study three-dimensional partition functions.
Another interesting example is the application of rank one elliptic case
\cite{Rosengren}, in which multiple commutation relations
were effectively used to derive transformation formulas
for elliptic hypergeometric series.
It would be interesting to explore the usage 
of the higher rank version to special functions,
not to mention partition functions and correlation functions. 
\section*{Acknowledgements}
K. M. thanks Hitoshi Konno for discussions.
We thank Jacques H.H. Perk for pointing out relevant references on the nested Bethe ansatz.
We are grateful to the referee for the careful reading of our manuscript and for the helpful suggestions.
This work was partially supported by Grant-in-Aid for Scientific Research (C) 24K06889. 
\color{black}
\section*{Declarations}

\subsection*{Conflict of Interest Statement}
The authors declare that they have no conflicts of interest relevant to the content of this article.

\subsection*{Data Availability Statement}
No new data were generated or analyzed in this study; therefore, data sharing is not applicable.

\subsection*{Funding}
This research was partially supported by
Grant-in-Aid for Scientific Research (C) 24K06889.

\section*{Appendix: A construction of the Gelfand-Tsetlin basis}
\renewcommand*{\thetheorem}{\mbox{\textrm A.\arabic{theorem}}}
\renewcommand*{\theequation}{\mbox{\textrm A.\arabic{equation}}}

We briefly review
the construction of the Gelfand-Tsetlin basis by Molev \cite{Molev},
applied to the case of tensor product of the vector representation
of the quantum affine algebra $U_q(\widehat{\mathfrak{gl}}_n)$.
Here, we take $q$ to be a generic complex number.
We also remark that
there is another way of construction due to Nazarov-Tarasov \cite{NT},
which the Gelfand-Tsetlin basis
is constructed using quantum minors instead of single $L$-operators
even for the case of
tensor product of the vector representation. One can also see
\cite{KM} for the case of the elliptic quantum group $U_{q,p}(\widehat{\mathfrak{gl}}_n)$.




We use the version of the $R$-matrix \eqref{trigRverone},
which is obtained by replacing $q$ by $q^{-1}$
in \cite[(2.15)]{HM}.

The quantum minor \cite[(2.28)]{HM}
adopted to our convention  is
\begin{align}
T_{b_1,\dots,b_r}^{a_1,\dots,a_r}(u)
=\sum_{\sigma \in S_r} (-q)^{-\ell(\sigma)}
T_{a_r b_{\sigma(r)}}(q^{2r-2}u) \cdots T_{a_1 b_{\sigma(1)}}(u),
\end{align}
for $r=1,\dots,N$ and $b_1<b_2<\cdots<b_N$.
$\ell(\sigma)$ is the length of the permutation $\sigma \in S_r$.

The quantum minor $T_{b_1,\dots,b_r}^{a_1,\dots,a_r}(u)$
satisfies
\begin{align}
[T_{b_1,\dots,b_r}^{a_1,\dots,a_r}(u), T_{a_j b_j}(v)]=0, \label{quantumminorcommutativity}
\end{align}
for $1 \le j \le r$.

We define the quantum determinant
\begin{align}
\mathrm{qdet} T^{(j)}(u):=T_{1,\dots,j}^{1,\dots,j}(u), \ \ \ j=1,\dots,N.
\label{quantumdeterminant}
\end{align}

There is a distinguished commutative subalgebra called the {\it Gelfand-Tsetlin subalgebra}
generated by $\mathrm{qdet} T^{(j)}(u)$, $j=1,\dots,N$.
A basis in which all these elements are simultaneously diagonalized is called the Gelfand–Tsetlin basis.
We introduce a partition $J=(J_1,\dots,J_N)$ as in \eqref{partitionJ}.

Let $\overline{w}=\{ w_1,\dots,w_n  \}$
and $\overline{w}_{J_j}=\{ w_k \ | \ k \in J_j  \}$ ($j=1,\dots,N$).
We take $T(u)$ to be $T^{\mathrm{vect}}(u;\overline{w})
=R_{0n}(u,w_n) \cdots R_{01}(u,w_1)
$ acting on $V_1 \otimes \cdots \otimes V_n$.

\begin{definition}
\begin{align}
\widehat{\xi}_J=
T_{21}(\overline{w}_{J_1}) T_{32}(\overline{w}_{J_1} \cup \overline{w}_{J_2}) \cdots T_{N N-1}(\overline{w}_{J_1} \cup \cdots \cup \overline{w}_{J_{N-1}}) e_{N^n}.
\end{align}
\end{definition}

Observe that $e_{N^n}$ is obtained in the case $J_1 = \dots = J_{N-1} = \varnothing$ and $J_N = \{1, \dots, n\}$.

\begin{definition}
We call that a vector $\eta$ is a singular vector of weight $\mu(u) = (\mu_1(u),\dots,\mu_k(u))$
with respect to the subalgebra $U_q(\widehat{\mathfrak{gl}}_n)$ if it satisfies
\begin{align}
T_{ij}(u) \cdot \eta&=0, \ \ \ 1 \le i < j \le k, \\
T_{ii}(u) \cdot \eta&=\mu_i(u) \eta, \ \ \ 1 \le i \le k.
\end{align}
\end{definition}

\begin{lemma} \label{singularvectorlemma}
Let $\eta$ be a singular vector of weight $\mu(u) = (\mu_1(u),\dots,\mu_k(u))$
with respect to the subalgebra $U_q(\widehat{\mathfrak{gl}}_n)$. Assume that $\eta$ satisfies 
$T_{kk}(\alpha) \cdot \eta = 0$ for some $\alpha \in \mathbb{C}^\times$.
Then $T_{k+1,k}(\alpha) \cdot \eta$ is also a singular vector with respect to 
$U_q(\widehat{\mathfrak{gl}}_n)$, with weight given by, for $u \neq \alpha$,
\begin{align}
\Bigg( \mu_1(u),\dots,\mu_{k-1}(u),\frac{qu-q^{-1}\alpha}{u-\alpha} \mu_k(u)  \Bigg).
\end{align}

\end{lemma}
\begin{proof}
From the $RTT$ relation \eqref{eq:RTT++}, we have the following commutation relations
\begin{align}
&(z_1-z_2)T_{ij}(z_1) T_{k+1,k}(z_2)+(q-q^{-1})z_2 T_{k+1,j}(z_1) T_{ik}(z_2)
\nonumber \\
=&(q-q^{-1})z_1 T_{k+1,j}(z_2) T_{ik}(z_1)+(z_1-z_2) T_{k+1,k}(z_2) T_{ij}(z_1), \ \ \ 1 \le i \le j \le k-1, \\
&(z_1-z_2)T_{ik}(z_1) T_{k+1,k}(z_2)+(q-q^{-1})z_2 T_{k+1,k}(z_1)T_{ik}(z_2) \nonumber \\
=&(qz_1-q^{-1}z_2)T_{k+1,k}(z_2) T_{ik}(z_1), \ \ \ 1 \le i \le k.
\end{align}

Using these relations, we can show
\begin{align}
T_{ij}(u)T_{k+1,k}(\alpha) \cdot \eta&=0, \ \ \ 1 \le i < j \le k, \\
T_{kk}(u)T_{k+1,k}(\alpha) \cdot \eta&=\frac{qu-q^{-1} \alpha}{u-\alpha} \mu_k(u) T_{k+1,k}(\alpha) \cdot \eta, \\
T_{ii}(u)T_{k+1,k}(\alpha) \cdot \eta&=\mu_i(u) T_{k+1,k}(\alpha) \cdot \eta, \ \ \ 1 \le i \le k-1.
\end{align}
\end{proof}

Define $\lambda_{jk}^J(u)$ $1 \le k \le j \le N$ by
\begin{align}
\lambda_{jk}^J(u)&= u-\overline{w}, \ \ \  j \neq k \\
\lambda_{jj}^J(u)
&=(u-\overline{w}_{J_{j+1}} \cup \cdots \cup \overline{w}_{J_N})
(qu-q^{-1} \overline{w}_{J_{1}} \cup \cdots \cup \overline{w}_{J_j}).
\end{align}

Using Lemma \ref{singularvectorlemma}, \eqref{quantumminorcommutativity}
and
\begin{align}
T_{ii}(u)e_{N^n}&=(u-\overline{w})e_{N^n}, \ \ \ i \neq N, \\
T_{NN}(u)e_{N^n}&=(qu-q^{-1} \overline{w})e_{N^n},
\end{align}
we can show the following.
\begin{proposition}
\label{Molevdiagonalization}
\begin{align}
\mathrm{qdet} T^{(j)}(u) \cdot \widehat{\xi}_J
=\prod_{k=1}^j \lambda_{jk}^J (q^{2k-2}u) \widehat{\xi}_J.
\end{align}
\end{proposition}
\begin{proof}
We apply the argument in \cite{Molev}.
Using Lemma \ref{singularvectorlemma} and \eqref{quantumminorcommutativity},
we can show
\begin{align}
\widehat{\xi}_J^{(j)}=
T_{j+1,j}(\overline{w}_{J_1} \cup \cdots \cup \overline{w}_{J_{j}}) 
 \cdots T_{N,N-1}(\overline{w}_{J_1} \cup \cdots \cup \overline{w}_{J_{N-1}}) e_{N^n},
\end{align}
is a singular vector of the $U_q(\widehat{\mathfrak{gl}}_j)$ satisfying
\begin{align}
L_{k \ell}(u) \cdot \widehat{\xi}_J^{(j)}&=0, \ \ \ 1 \le k < \ell \le j, \\
L_{k k}(u) \cdot \widehat{\xi}_J^{(j)}&=\lambda_{kk}^J(u) \widehat{\xi}_J^{(j)}, \ \ \ k=1,\dots,j.
\end{align}
Together with the definition of $\mathrm{qdet} T^{(j)}(u)$, we get
\begin{align}
\mathrm{qdet} T^{(j)}(u) \cdot \widehat{\xi}_J^{(j)}
=\prod_{k=1}^j \lambda_{jk}^J (q^{2k-2}u) \widehat{\xi}_J^{(j)}. \label{proofgzone}
\end{align}
From \eqref{quantumminorcommutativity}, we have
\begin{align}
[ \mathrm{qdet} T^{(j)}(u), T_{k+1,k}(v) ]=0, \ \ \ 1 \le k <j. \label{proofgztwo}
\end{align}
Using \eqref{proofgzone}, \eqref{proofgztwo} and
\begin{align}
\widehat{\xi}_J=
T_{21}(\overline{w}_{J_1})
\cdots
T_{j j-1}(\overline{w}_{J_1} \cup \cdots \cup \overline{w}_{J_{j-1}}) 
\widehat{\xi}_J^{(j)},
\end{align}
the claim follows.
\end{proof}

%





\begin{thebibliography}{00}
%
%
\bibitem{Drinfeld}
V.G. Drinfeld,
{\it A New realization of Yangians and quantized affine algebras},
Sov. Math. Dokl. 36 (1988) 212-216.

\bibitem{Jimbo}
M. Jimbo,
{\it
A $q$-difference analogue of $U(\mathfrak{g})$ and the Yang-Baxter equation},
Lett. Math. Phys. 10 (1985) 63-69.


\bibitem{RTF} N.Y. Reshetikhin, L.A. Takhtadzhyan and L.D. Faddeev, 
{\it Quantization of Lie groups and Lie algebras},
Algebra Anal. 1 (1989) 178-207.



\bibitem{Bethe}
H. Bethe,
{\it Zur Theorie der Metalle},
Z. Phys. 71 (1931) 205-226.


\bibitem{PS}
J.H.H. Perk  and C.L. Schultz,
{\it Families of commuting transfer matrices in q-state vertex models},
Non-linear Integrable system—Classical theory and quantum theory (Proceedings of RIMS symposium, Kyoto, Japan, 13–16 May 1981) ed. by M. Jimbo and T. Miwa, World Scientific, Singapore, 1983, 326–343.

\bibitem{Baxterbook} R.J. Baxter, 
{\it Exactly Solved Models in Statistical Mechanics}, Academic Press, London, 1982.

\bibitem{KBI}
V.E. Korepin, N.M. Bogoliubov and A.G. Izergin, {\it
Quantum Inverse Scattering Method and Correlation Functions}, Cambridge University Press, Cambridge, 1993.

\bibitem{Slavnovbook}
N. Slavnov, {\it Algebraic Bethe Ansatz And Correlation Functions: An Advanced Course},
World Scientific, 2022.

\bibitem{KMST}
N. Kitanine, J.M. Maillet, N.A. Slavnov and V. Terras,
{\it
Spin–spin correlation functions of the XXZ-1/2 Heisenberg chain in a magnetic field},
Nucl. Phys. B 641 (2002) 487-518.

\bibitem{ABPW}
A. Aggarwal, A. Borodin, L. Petrov and M. Wheeler,
{\it
Free Fermion Six Vertex Model: Symmetric Functions and Random Domino Tilings},
Selecta Mathematica 29 (2023) 36.

\bibitem{IMO}
S. Iwao, K. Motegi and R. Ohkawa,
{\it Tetrahedron equation and Schur functions},
J. Phys. A: Math. Theor. 58 (2025) 015201.




\bibitem{PRS}
S. Pakuliak, E. Ragoucy and  N.A. Slavnov,
{\it
Bethe vectors of quantum integrable models based on $U_q(\mathfrak{gl}_N)$},
J. Phys. A 47 (2014) 105202.



\bibitem{HLPRS}
A.A. Hutsalyuk, A. Liashyk, S.Z. Pakuliak, E. Ragoucy and N.A. Slavnov,
{\it
Current presentation for the super-Yangian double $DY(\mathfrak{gl}(m|n))$ and Bethe vectors},
Russ. Math. Surv. 72, (2017) 33.

\bibitem{LP}
 A. Liashyk and S.Z. Pakuliak,
 {\it
Recurrence relations for off-shell Bethe vectors in trigonometric integrable models},
J. Phys. A: Math. Theor. 55 (2022) 075201.


\bibitem{Schultz}
C.L. Schultz,
{\it Eigenvectors of the multi-component generalization of the six-vertex model},
Physica A 122 (1983) 71-88.




\bibitem{Reshetikhin}
N.Y. Reshetikhin,
{\it Calculation of the norm of bethe vectors in models with SU(3)-symmetry},
J. Sov. Math. 46 (1989) 1694-1706.




\bibitem{TVleningrad}
V. Tarasov, A. Varchenko, 
{\it
Jackson integral representations for solutions of the Knizhnik–Zamolodchikov quantum equation},
Leningrad Math. J. 6 (1994) 275-313.

\bibitem{Mimachi}
K. Mimachi, 
{\it
A solution to quantum Knizhnik–Zamolodchikov equations and its application to eigenvalue problems of the Macdonald type},
Duke Math. J. 85 (1996) 635-658.


\bibitem{TVasterisque}
V. Tarasov, A. Varchenko, 
{\it
Geometry of $q$-hypergeometric functions, quantum
affine algebras and elliptic quantum groups},
Astérisque 246 (1997).



\bibitem{FKPR}
L. Frappat, S. Khoroshkin, S. Pakuliak and E. Ragoucy,
{\it Bethe Ansatz for the Universal Weight Function},
Ann. Henri Poincaré, 10 (2009) 513.




\bibitem{BPR}
S. Belliard, S. Pakuliak and E. Ragoucy,
{\it
Bethe Ansatz and Bethe Vectors Scalar Products},
SIGMA 6 (2010) 094.

\bibitem{TVsigma}
V. Tarasov and A. Varchenko,
{\it Combinatorial Formulae for Nested Bethe Vectors},
SIGMA 9 (2013) 048.

\bibitem{FM}
O. Foda and M. Manabe,
{\it Nested coordinate Bethe wavefunctions from the Bethe/Gauge correspondence},
 J. High Energy Phys. 2019 (2019) 36.

\bibitem{GR}
A. Gerrard, V. Regelskis, 
{\it
Nested algebraic Bethe ansatz for deformed orthogonal and symplectic spin chains}, 
Nucl. Phys. B 956 (2020) 115021.


\bibitem{Slavnov}
N. Slavnov,
{\it Introduction to the nested algebraic Bethe ansatz},
SciPost Phys. Lect. Notes (2020) 19.



\bibitem{BW}
A. Borodin and M. Wheeler, 
{\it Colored Stochastic Vertex Models and Their Spectral Theory},
Astérisque 437 (2022).


\bibitem{GMS}
A.J. Gerrard, K. Motegi and K. Sakai,
{\it Higher rank elliptic partition functions and multisymmetric elliptic functions},
Nucl. Phys. B 1011 (2025) 116805.

\bibitem{KT}
M. Kosmakov and V. Tarasov,
{\it
New combinatorial formulae for nested Bethe vectors II},
Lett. Math. Phys. 115 (2025) 12.


\bibitem{NT}
M. Nazarov and V. Tarasov,
{\it
Representations of Yangians with Gelfand-Zetlin bases},
Journal fur die Reine und Angewandte Mathematik 496 (1998) 181-212.



\bibitem{Molev}
A. Molev,
{\it Yangians and classical Lie algebras}, Mathematical Surveys and Monographs, vol. 143, American Mathematical Society, Providence, RI,. 2007.


\bibitem{KonnoJintone}
H. Konno, 
{\it
Elliptic weight functions and elliptic $q$-KZ equation},
J. Int. Syst. 2 (2017) xyx011.


\bibitem{KonnoJinttwo}
H. Konno, 
{\it
Elliptic stable envelopes and finite-dimensional representations of
elliptic quantum group},
J. Int. Syst. 3 (2018) xyy012.

\bibitem{RTV}
R. Rimányi, V. Tarasov, A. Varchenko,
{\it
Elliptic and K-theoretic stable envelopes and Newton polytopes},
Sel. Math. 25 (2019) 16.




\bibitem{KM}
H. Konno and K. Motegi, {\it
Gelfand-Tsetlin Bases for Elliptic Quantum Groups},
arXiv:2408.09712.

\bibitem{Korepin}
V.E. Korepin,
{\it Calculation of Norms of Bethe Wave Functions},
Commun. Math. Phys. 86 (1982) 391-418.

\bibitem{Izergin}
A. Izergin, 
{\it Partition function of the six-vertex model in a finite volume},
Sov. Phys. Dokl. 32 (1987) 878.

\bibitem{HM}
M.J. Hopkins and A.I. Molev,
{\it
A $q$-analogue of the centralizer construction and skew representations of the quantum
affine algebra,}
SIGMA 2 (2006) 092.

\bibitem{Rosengren}
H. Rosengren,
{\it
Felder’s Elliptic Quantum Group and Elliptic Hypergeometric Series on the Root System $A_n$},
Int. Math. Res. Not. 2011 (2011) 2861–2920.


\bibitem{DF}
J. Ding and I. B. Frenkel, 
{\it Isomorphism of Two Realizations of Quantum Affine Algebra $U_q(\widehat{\mathfrak{gl}(n)})$}, 
Commun. Math. Phys. 156 (1993) 277-300.

\bibitem{KK}
T. Kojima and H. Konno, 
{\it The Elliptic Algebra and the Drinfeld Realization of the Elliptic Quantum Group},
Commun. Math. Phys. 239, (2003) 405–447.


\end{thebibliography}
\end{document}